\documentclass[11pt]{article}
\usepackage[normalem]{ulem}

\usepackage[title, titletoc]{appendix}

\usepackage[margin=1in]{geometry}
\usepackage[utf8]{inputenc}
\usepackage{enumerate}
\usepackage{amsfonts}
\usepackage{amsmath}
\usepackage{amssymb}
\usepackage{mathrsfs}
\usepackage{bbm}
\usepackage{amsthm}
\usepackage{xcolor}
\usepackage{multirow}
\usepackage{indentfirst}
\usepackage{comment}

\theoremstyle{definition}

\newtheorem{theorem}{Theorem}[section]
\newtheorem{example}[theorem]{Example}
\newtheorem{definition}[theorem]{Definition}
\newtheorem{corollary}[theorem]{Corollary}
\newtheorem{proposition}[theorem]{Proposition}
\newtheorem{lemma}[theorem]{Lemma}
\newtheorem{remark}[theorem]{Remark}

\numberwithin{equation}{section}

\newcounter{cnstcnt}
\newcommand{\newconstant}{
    \refstepcounter{cnstcnt}
    \ensuremath{c_{\thecnstcnt}}}
\newcommand{\oldconstant}[1]{\ensuremath{c_{\ref{#1}}}}

\newcounter{integer}

\usepackage{hyperref}
\hypersetup{
    colorlinks = true,
    linkcolor = blue,
    filecolor = blue,      
    urlcolor = blue,
    citecolor = blue
}

\usepackage{enumerate}
\usepackage{enumitem}

\makeatletter
\def\@tocline#1#2#3#4#5#6#7{\relax
  \ifnum #1>\c@tocdepth 
  \else
    \par \addpenalty\@secpenalty\addvspace{#2}%
    \begingroup \hyphenpenalty\@M
    \@ifempty{#4}{%
      \@tempdima\csname r@tocindent\number#1\endcsname\relax
    }{%
      \@tempdima#4\relax
    }%
    \parindent\z@ \leftskip#3\relax \advance\leftskip\@tempdima\relax
    \rightskip\@pnumwidth plus4em \parfillskip-\@pnumwidth
    #5\leavevmode\hskip-\@tempdima
      \ifcase #1
       \or\or \hskip 1em \or \hskip 2em \else \hskip 3em \fi%
      #6\nobreak\relax
    \hfill\hbox to\@pnumwidth{\@tocpagenum{#7}}\par
    \nobreak
    \endgroup
  \fi}
\makeatother

\allowdisplaybreaks

\title{Universality of Annealed Complexity \\
in the Spherical Pure $p$-Spin Model}
\author{Wei-Kuo Chen\thanks{School of Mathematics, University of Minnesota, Email: wkchen@umn.edu.} \and  Te-Lun Lu \thanks{School of Mathematics, University of Minnesota, Email: lu000646@umn.edu.}  \and Arnab Sen \thanks{School of Mathematics, University of Minnesota, Email: arnab@umn.edu.}} 
\date{\today}

\begin{document}

\maketitle

\begin{abstract}
    In an influential work \cite{auffinger2013random}, Auffinger, Ben Arous, and {\v{C}}ern{\`y} investigated the limiting annealed complexity in the spherical pure $p$-spin model associated to the Gaussian disorder. Their results, in particular, included an explicit expression for the total number of critical points of the Hamiltonian below any given energy level. In this paper, we consider the same anneal complexity associated to a disorder interpolating between an arbitrary random variable and its Gaussian counterpart. Assuming strict subgaussianity of the former, our result shows that the annealed complexity shares the same limit as in the Gaussian case.
\end{abstract}


\section{Introduction and Main Results}

\noindent The study of the complexity in the spin glass theory is concerned about the total numbers of critical or saddle points for the disordered Hamiltonians. By analyzing their asymptotic behavior, they help to understand the geometry of the high-dimensional energy landscapes. Earlier results were pioneered in the beautiful work of Auffinger-Ben Arous-{\v{C}}ern{\`y} \cite{auffinger2013random} in the sphereical pure $p$-spin model, whose Hamilonian is defined as 
\[H_N^G(\sigma) = \dfrac{1}{N^{(p - 1)/2}}\sum_{I \in [N]^p} G_I \sigma_I\]
for $\sigma\in \sqrt{N}\mathbb{S}^{N - 1} = \{\sigma \in \mathbb{R}^N \colon \|\sigma\|_2^2 = N\},$
where $\sigma_I := \prod_{\alpha = 1}^p \sigma_{i_\alpha}$ for $I=(i_1,\ldots,i_p).$ 
Denote by $\nabla_S$ and $\nabla_S^2$ the spherical gradient and Hessian, respectively. Let
\[\text{Crt}_G((-\infty,u)) := \big\{\sigma\in \sqrt{N}\mathbb{S}^{N-1} \colon \nabla_SH_N^G(\sigma) = 0, H_N^G(\sigma) < Nu\big\}.\]
be the collection of all
critical points below the energy level $u\in \mathbb{R}$.
In \cite[Theorem 2.8]{auffinger2013random}, it was established
an explicit formula for the logarithmic asymptotics of the expected number of critical point, i.e.,
\begin{equation}\label{Equation:ABC_Main} \lim_{N \to \infty} \dfrac{1}{N}\log \mathbb{E}\big[\text{Crt}_G((-\infty, u))\big] = \begin{cases} \frac{1}{2}\log(p - 1) + \frac{1}{2}+ \Omega\big(\frac{2u}{E_\infty}) - \frac{u^2}{2}, & u \leq -E_\infty, \\ \frac{1}{2}\log(p - 1) - \frac{p - 2}{4(p - 1)}u^2, & -E_\infty \leq u \leq 0, \\ \frac{1}{2}\log(p - 1), & u \geq 0, \end{cases} \end{equation}
where
\begin{equation}\label{Equation:Energy}
    E_\infty = 2\sqrt{\frac{p - 1}{p}},\qquad
    \mu_{\mathrm{sc}}(d\lambda) = \frac{\sqrt{4-\lambda^2}}{2\pi}\mathbbm{1}_{\{|\lambda|\leq 2\}}\,d\lambda
\end{equation}
and $\Omega(x)$ is the log-potential of semicircular law given by
\begin{equation}\label{Equation:SC_Potential} 
    \Omega(x) = \int_{\mathbb{R}} \log|\lambda - x|\,\mu_{\text{sc}}(d\lambda) = \begin{cases}
        \frac{x^2}{4} - \frac{1}{2} - \frac{|x|}{4}\sqrt{x^2 - 4} + \log\Big(\frac{|x| + \sqrt{x^2 - 4}}{2}\Big), & |x| \geq 2, \\
        \frac{x^2}{4} - \frac{1}{2}, & |x| \leq 2.
    \end{cases} 
\end{equation}
In addition to this, the annealed complexity of saddle points was obtained in the same work \cite{auffinger2013random} and later generalized to the mixed $p$-spin model by Auffinger-Ben Arous \cite{Auffinger13}. Based on these, the study of complexity has received growing attention and many results have been made available. First, the second moment calculation of $\mbox{Crt}_G((-\infty,u))$ was carried out in the pure $p$-spin model by Subag \cite{subag2017complexity}, which concluded that the quenched and annealed complexities of the critical points indeed match with each other. More sophisticated formulas for the annealed complexity were also obtained in some Gaussian fields with isotropic increments by Auffinger-Zeng \cite{auffingerZeng2023complexity} as well as with non-rotationally invariant potential by Chen-Lu-Sen \cite{chen2026complexity}.

While the aforementioned results heavily rely on the Gaussianity of the Hamiltonians, it is a natural question to ask if the annealed complexity \eqref{Equation:ABC_Main} will remain the same or not when $G_I$'s are not Gaussian. Let $\{J_I\}_{I \in [N]^p, N \in \mathbb{N}}$ be a collection of independent random variables independent of $G_I$'s. By switching $G$ with $J$, let $H_N^J$ be defined in the same manner as $H_N^G$. 
A reasonable setting to investigate the universality of the annealed complexity is to consider the pure $p$-spin Hamiltonian associated to the interpolated disorder between $G_I$ and $J_I$,
\begin{equation}\label{Equation:Hamiltonian_Model}
    H_t(\sigma) :=  \dfrac{1}{N^{(p - 1)/2}}\sum_{I \in [N]^p} \big(\sqrt{1 - t}G_I + \sqrt{t}J_I\big) \sigma_I
\end{equation}
for $0\leq t\leq 1$ and investigate the corresponding number of critical points,
\begin{equation}\label{Equation:Complexity}
    \text{Crt}_t(B) := \big|\big\{\sigma \in \sqrt{N}\mathbb{S}^{N - 1} \,\big\vert\, H_t(\sigma) \in NB, \nabla_SH_t(\sigma) = 0\}\big|
\end{equation}
for all Borel set $B \subseteq \mathbb{R}$. In this paper we show that the annealed complexity does exhibit universality for all $0<t<1$ if $J_I$'s are strictly subgaussian. Here we recall that a random variable $X$ (of unit variance)  is $\kappa_{\text{SG}}^2$-subgaussian if
\begin{equation}\label{Equation:SG}
    \mathbb{E}\big[e^{\lambda X}\big] \leq e^{\kappa_{\text{SG}}^2\lambda^2/2}, \quad\forall\lambda \in \mathbb{R}. \tag{SG}
\end{equation}
It is strict subgaussian if
\begin{equation}\label{Equation:SSG}
    \kappa_{\text{SG}} = 1. \tag{SSG}
\end{equation}

The following is our main result.

\begin{theorem}\label{Theorem:Universality_Main}
    Let $p \geq 3$. Let $\{J_I\}_{I \in [N]^p, N \in \mathbb{N}}$ be independent with zero mean and unit variance. If they satisfy \eqref{Equation:SSG}, then for any $0 \leq t< 1$.
    \begin{equation}\label{universality:eq1}
        \lim_{N \to \infty} \dfrac{1}{N}\log \mathbb{E}\big[\text{Crt}_t((-\infty, u))\big] = \lim_{N \to \infty} \dfrac{1}{N}\log \mathbb{E}\big[\text{Crt}_0((-\infty, u))\big].
    \end{equation}
\end{theorem}

\begin{example}
   Assume that $J_I$'s are i.i.d.\ symmetric Racemacher, i.e., $\mathbb{P}(J_I = \pm 1)  = 1/2$. It is centered with unit variance. Moreover, its moment generating function satisfies \eqref{Equation:SSG} since $\mathbb{E}\big[e^{\lambda J}\big] = \cosh(\lambda) \leq e^{\lambda^2/2}$. Then the conclusion of Theorem \ref{Theorem:Universality_Main} holds.
\end{example}

\begin{example}
    Assume that $J_I$'s are i.i.d. sampled uniformly from $(-\sqrt{3}, \sqrt{3})$. It is centered with unit variance. Moreover, its moment generating function satisfies \eqref{Equation:SSG} since $\mathbb{E}\big[e^{\lambda J}\big] = \sinh(\sqrt{3}\lambda)/(\sqrt{3}\lambda) \leq e^{\lambda^2/2}$. Thus, \eqref{universality:eq1} holds.
\end{example}

\begin{remark}\label{rem:gaussian_perb}
    When $t = 1$, $H_1=H_N^J$ and the statement \eqref{universality:eq1} is not necessarily true. Indeed, when $J_I$'s are discrete, $H_N^J$ may possess nonisolated critical points. For example, for the symmetric Rademacher disorder, $J_I$'s satisfy all assumptions in Theorem \ref{Theorem:Universality_Main}. However, on the positive-probability event $\{J_I = 1 \text{ for all }I\in[N]^p\}$, we have
    \[H_N^J(\sigma)=N^{-(p-1)/2}\Big(\sum_{i=1}^N\sigma_i\Big)^p\]
   and its spherical gradient and Hessian vanish on the equator ${\sum_{i = 1}^N\sigma_i=0}$. Thus, the Hamiltonian fails to be a Morse function with positive probability and $\mathbb{E}\big[\text{Crt}_1(\mathbb{R})\big] = \infty$.   In addition, Kim~\cite{kim2026regularized} recently showed that, for every $p \geq 3$ and $\sigma > 1$, one can construct a symmetric random variable $X$ with the density $\mu = \varepsilon \texttt{Normal}(0,\sigma^2) + (1-\varepsilon)\nu$, where $\varepsilon\in(0,1)$ is sufficiently small and $\nu$ is a symmetric, compactly supported $C^\infty$ density, whose first $2p$ moments match those of the standard Gaussian random variable. In particular, $X$ has zero mean, unit variance, and a strictly positive $C^\infty$ density. It is subgaussian, but not strictly subgaussian. If $J_I$'s are i.i.d. sampled from $X$, Kim proved that there exists a compact interval $B\subset(0,\infty)$ such that
   \[\liminf_{N\to\infty}\frac{1}{N}\log\mathbb E_\mu\big[\operatorname{Crt}_N^J(B)\big] > \lim_{N\to\infty} \frac{1}{N} \log\mathbb E\big[\operatorname{Crt}_N^G(B)\big].\]
   Thus, even Gaussian moment matching through order $2p$, together with smoothness, and subgaussianity, does not by itself imply universality of the unperturbed annealed complexity.
\end{remark}

Our proof of Theorem \ref{Theorem:Universality_Main} starts with the standard Kac-Rice formula for the annealed complexity. Note that even for smooth non-Gaussian disorder, conditioning on a vanishing gradient makes the conditional law of the Hessian and the density of the gradient difficult to analyze. By the virtue of the form of the disorder $\sqrt{1 - t}G_I + \sqrt{t}J_I$, it allows us to avoid this obstruction by first conditioning on $J$ and then applying Kac-Rice formula to the remaining Gaussian field, for which the conditional Hessian is a deterministic matrix plus an independent GOE perturbation. We then take average over $J$ and control the remaining dependence through some concentration and large-deviation estimates. 

While we utilize several known estimates and results from \cite{ben2023exponential,ho2025hessian} to implement the above approach, the main key ingredient of this paper is played by the following joint large deviation principle (LDP), which holds under the strict subgaussianity assumption \eqref{Equation:SSG}.  Let $B(\sigma)$ stand for a semi-orthogonal matrix representing the orthogonal projection from $\mathbb R^N$ onto the tangent space of $\sqrt N\mathbb{S}^{N - 1}$ at $\sigma$, i.e., $B(\sigma)$ satisfies 
\begin{align}
    B(\sigma)^{\top}B(\sigma) &= I_{N - 1}, \label{Equation:Projection_1} \\
    B(\sigma)B(\sigma)^{\top} &= I_N -\frac{\sigma\sigma^\top}{N} =:P(\sigma), \label{Equation:Projection_2}
\end{align}
where $I_{N-1}$ and $I_N$ the $(N-1)$- and $N$-dimensional identity matrices, respectively.

\begin{theorem}\label{Theorem:LDP}
    Let $p \geq 3$. Let $\{J_I\}_{I \in [N]^p, N \in \mathbb{N}}$ be independent  with zero mean and unit variance. Assume that they satisfy \eqref{Equation:SSG}. Under the product measure  $\mathbb{P}_J \otimes \mathbb{P}_\sigma$ for $\sigma \sim \mathbb{P}_\sigma := \text{Unif}(\sqrt{N}\mathbb{S}^{N-1})$, the pair
    \begin{equation}\label{Equation:Joint_2}
        \Big(\dfrac{H_N^J(\sigma)}{\sqrt{N(N - 1)}}, \dfrac{\|B(\sigma)^\top\nabla H_N^J(\sigma)\|_2^2}{N}\Big)
    \end{equation}
    satisfies the LDP  at speed $N$ with the good rate function 
    \[\mathfrak{I}_2(a, b) = \begin{cases}
        \frac{a^2}{2} + \frac{b - p}{2p} + \frac{1}{2}\log\big(\frac{p}{b}\big), &a\in \mathbb{R}, b > 0, \\
        \infty, &a\in \mathbb{R}, b \leq 0.
    \end{cases}\]
\end{theorem}

\begin{remark}
    \rm Theorem \ref{Theorem:LDP} is an LDP universality result under \eqref{Equation:SSG}.
    Interestingly, the same strict subgaussianity assumption also appears in random matrix theory.  In \cite{GuionnetHusson2020}, Guionnet and Husson showed that the largest eigenvalue of a Wigner matrix with entries satisfying \eqref{Equation:SSG} obeys an LDP at speed $N$ with the same rate function as in the Gaussian case.  On the contrary, Augeri, Guionnet, and Husson \cite{AugeriGuionnetHusson2021} showed that this LDP universality can fail when the entries are subgaussian but not strictly subgaussian.  We anticipate an analogous failure of universality for the joint LDP of the pair in \eqref{Equation:Joint_2} if \eqref{Equation:SSG} is weakened to mere subgaussianity.
\end{remark}

\begin{remark}
\rm In this work, we only consider the pure $p$-spin model for the clarity of exposition. Indeed, following our argument, it can be shown that Theorems \ref{Theorem:Universality_Main} and \ref{Theorem:LDP} hold also for the mixed $p$-spin model. 
\end{remark}

\begin{remark}
    Before the completion of this paper, we learned of a concurrent work by Kivimae and McKenna \cite{Kiv2026}. They study the same interpolated model assuming that the non-Gaussian disorders have general subexponential tails. Their main result establishes the following lower bound:
    \[\liminf_{N \to \infty} \frac{1}{N}\log \mathbb{E}\big[\mathrm{Crt}_t((-\infty, u))\big] \geq \lim_{N \to \infty} \frac{1}{N}\log \mathbb{E}\big[\mathrm{Crt}_0((-\infty, u))\big]\]
    and it holds true also in the mixed $p$-spin models. As discussed in the preceding remark, this suggests a potentially interesting dichotomy, analogous to those observed in \cite{AugeriGuionnetHusson2021, cook2023full, GuionnetHusson2020}: universality may hold under strict subgaussianity, but need not hold under subgaussianity alone.
\end{remark}

The rest of the paper is organized as follows. Section \ref{sec:notation} introduces some notation. In Section~\ref{Section:First_Moment}, we apply the Kac-Rice formula to derive the main integral representation for the expected complexity. This computation is standard and closely follows  \cite{auffinger2013random}. Section \ref{Section:Isolation} is devoted to separating the determinant of the Hessian term in the Kac-Rice formula. It relies on the technical lemmas prepared in Section \ref{Section:Property}, where we establish superexponential concentration estimates and universality results for the empirical spectral measure of the Hessian matrix. In Section \ref{Subsection:MGF}, we study the joint moment generating function for the pair \eqref{Equation:Joint_2}, where strict subgaussianity plays a crucial role. Based on this, Section \ref{Section:Main} carries out the proofs of the large deviation principle in Theorem \ref{Theorem:LDP} as well as the asserted universality, Theorem \ref{Theorem:Universality_Main}.

\bigskip

{\noindent \bf Acknowledgements.}
W.-K. Chen is partly supported by NSF grants DMS-2246715 and Simons Foundation MPS-TSM-00025876; T.-L. Lu is partly supported by NSF grants DMS-2246715; both thank the National Center for Theoretical Sciences in Taiwan for their hospitality during summer 2026, where part of this research was conducted. A. Sen is partly supported by Simons Foundation MP-TSM-00002716. All authors warmly thank E. Subag, P. Kivimae, and B. Mckenna for the illuminating discussions. 
 In the proof of the upper bound in Proposition \ref{Proposition:General_Bound}(i), the use of the Brascamp--Lieb inequality was suggested by ChatGPT (OpenAI).
\section{Notation}\label{sec:notation}
\noindent We introduce some notation that will be used for the remainder of this paper.

\begin{itemize}
    \item For a real number $x$, denote by $x_+ = \max\{x, 0\}$.
    \item Let $\langle \cdot, \cdot \rangle$ and $\|\cdot\|_2$ denote the standard inner product and Euclidean norm on $\mathbb{R}^N$, respectively.
    \item For an real-valued $N \times N$ matrix $A$, denote by $\|A\|_{\text{F}} = \big(\sum_{i, j = 1}^N A_{ij}^2\big)^{1/2}$ the Frobenius norm of $A$. Let $\text{Spec}(A)$ denote its spectrum, with eigenvalues counted according to their algebraic multiplicities, and let $\|A\|_{\text{op}}$ denote its operator norm. When $A$ is symmetric, we order its eigenvalues as $\lambda_1^\uparrow(A) \leq \cdots \leq \lambda_N^\uparrow(A)$ and define its empirical spectral measure by $\mu_A = N^{-1}\sum_{j = 1}^N \delta(\lambda_j^\uparrow(A))$, where $\delta(x)$ denotes the Dirac measure at $x$.
    \item For two probability measures $\mu,\nu$ on $\mathbb{R}$ with finite first moments, define the bounded Lipschitz distance by
    \[d_{\text{BL}}(\mu, \nu) = \sup_{\|f\|_{\mathrm{Lip}} \leq 1, \|f\|_\infty \leq 1} \int_{\mathbb{R}} f(x)\,(\mu - \nu)(dx),\]
    where $\|f\|_{\text{Lip}}$ and $\|f\|_\infty$ denotes the Lipschitz seminorm and supremum norm of the function $f \colon \mathbb{R} \to \mathbb{R}$, respectively.
    \item We write $f(N) = O_N(g(N))$ if there exists an absolute constant $C > 0$ such that
    $|f(N)|\leq Cg(N)$ for all sufficiently large $N$. Similarly, $f(N) = o_N(g(N))$ if
    $|f(N)/g(N)| \to 0$ as $N \to \infty$.
\end{itemize}

\section{First Moment of Complexity}\label{Section:First_Moment}

\noindent Recall that $B(\sigma)$ is a $N \times (N - 1)$ semi-orthogonal matrix representing the orthogonal projection from $\mathbb R^N$ onto the tangent space of $\sqrt N\mathbb{S}^{N - 1}$ at $\sigma$. We can then represent the spherical gradient and Hessian at $\sigma \in \sqrt{N}\mathbb{S}^{N - 1}$ by
\begin{align*}
    \nabla_S H_t(\sigma) &= B(\sigma)^\top\nabla H_t(\sigma), \\
    \nabla_S^2H_t(\sigma) &= B(\sigma)^\top\Big(\nabla^2H_t(\sigma) - \dfrac{pH_t(\sigma)}{N}I_N\Big)B(\sigma),
\end{align*}
where $\nabla$ and $\nabla^2$ denoted the ordinary Euclidean gradient and Hessian. Last, recall the Gaussian orthogonal ensemble (GOE) is the centered Gaussian matrix $(g_{ij})_{1 \leq i, j \leq N}$ with covariance structure
\[\mathbb{E}[g_{ij}g_{k\ell}] = \dfrac{\delta_{ik}\delta_{j\ell} + \delta_{i\ell}\delta_{jk}}{N}.\]

\subsection{Covariance Structure}

\begin{proposition}[Covariance Profile for Spherical Derivative]\label{Proposition:Conditional_Covariance}
	For any fixed $\sigma \in \sqrt{N}\mathbb{S}^{N - 1}$, conditioning on $J_I$'s, $(H_t(\sigma),\nabla_S H_t(\sigma), \nabla_S^2 H_t(\sigma))$ is a Gaussian field with the following means and covariances:
	\begin{enumerate}[label = (\roman*)]
		\item The Hamiltonian $H_t(\sigma) \sim \texttt{Normal}\big(\sqrt{t}H_N^J(\sigma), (1 - t)N\big)$.
		\item The gradient $\nabla_S H_t(\sigma) \sim \texttt{Normal}\big(\sqrt{t}B(\sigma)^\top \nabla H_N^J(\sigma), (1 - t)pI_{N - 1}\big)$.
		\item The Hessian satisfies
        \[\mathbb{E}\big[\nabla_S^2H_t(\sigma)\big] = \sqrt{t}B(\sigma)^\top\Big(\nabla^2H_N^J(\sigma) - \dfrac{pH_N^J(\sigma)}{N}I_N\Big)B(\sigma).\]
        Moreover, for all $i, j, k, \ell \in [N - 1]$,
		\begin{align*}
		      &\text{Cov}(\nabla_S^2 H_t(\sigma)_{ij}, \nabla_S^2 H_t(\sigma)_{k\ell}) = (1 - t)N^{-1}\big[p(p - 1)(\delta_{ik}\delta_{j\ell} + \delta_{i\ell}\delta_{jk}) + p^2\delta_{ij}\delta_{k\ell}\big].
	    \end{align*}
		\item For all $i, j, k \in [N - 1]$,
        \begin{align*}
            \text{Cov}\big(H_t(\sigma), \nabla_SH_t(\sigma)\big) &= 0, \\
            \text{Cov}\big(H_t(\sigma), \nabla_S^2H_t(\sigma)\big) &= -(1 - t)pI_{N - 1}, \\
            \text{Cov}\big(\nabla_SH_t(\sigma)_i, \nabla_S^2H_t(\sigma)_{jk}\big) &= 0.
        \end{align*}
	\end{enumerate}
    Additionally, under the conditional measure $\mathbb{P}(\cdot \mid \nabla_SH_t(\sigma) = 0)$, we have the joint law 
    \begin{align*}
        H_t(\sigma) & \overset{d}{=} \sqrt{\dfrac{N(N - 1)(p - 1)}{p}} x_N^{J, Z}(\sigma) \\
        \nabla_S^2H_t(\sigma) &\overset{d}{=} \sqrt{\dfrac{(N - 1)p(p - 1)}{N}}\big(M_{N - 1}^{J, \text{GOE}}(\sigma) - x_N^{J, Z}(\sigma)I_{N - 1}\big)
    \end{align*}
    where 
    \begin{align}
        x_N^{J, Z}(\sigma) &:= \sqrt{\dfrac{p}{p - 1}} \Big(\sqrt{t}\dfrac{H_N^J(\sigma)}{\sqrt{N(N - 1)}} + \sqrt{1 - t}\dfrac{Z}{\sqrt{N - 1}}\Big), \label{Equation:Def_x} \\
        M_{N - 1}^{J, \text{GOE}}(\sigma) &:= \sqrt{t}\sqrt{\dfrac{N}{N - 1}}\dfrac{B(\sigma)^\top\nabla^2H_N^J(\sigma)B(\sigma)}{\sqrt{p(p - 1)}} + \sqrt{1 - t} \cdot \text{GOE}_{N - 1}, \label{Equation:Def_M}
    \end{align}
    and $Z \sim \texttt{Normal}(0, 1)$. 
\end{proposition}

\begin{proof}
    These are straightforward computations, see \cite[Lemma 3.2]{auffinger2013random}.
\end{proof}

\subsection{Kac-Rice Formula}

\begin{proposition}
For every Borel set $S\subseteq\mathbb{R}$ and every $t\in[0,1)$, we have
    \begin{equation}\label{Equation:Kac_Rice}
        \begin{aligned}
            \mathbb{E}_{J, G}\big[\text{Crt}_t(S)\big] &= e^{N\big(\frac{1}{2}\log(p - 1) - \frac{1}{2}\log(1 - t) + \frac{1}{2}\big) + o_N(N)}  \\
            &\hspace{-2cm} \cdot \mathbb{E}_{\sigma, J, Z, \text{GOE}}\Big[e^{-\frac{Nt}{2p(1 - t)}y_N^J(\sigma)}\big|\det\big(M_{N - 1}^{J, \text{GOE}}(\sigma) - x_N^{J, Z}(\sigma)I_{N - 1}\big)\big| \cdot \mathbbm{1}_{\big\{x_N^{J, Z}(\sigma) \in \sqrt{\frac{Np}{(N - 1)(p - 1)}}S\big\}}\Big], 
        \end{aligned}
    \end{equation}
    where
    \begin{equation}\label{Equation:Def_y}
        y_N^J(\sigma) = N^{-1}\|B(\sigma)^\top\nabla H_N^J(\sigma)\|_2^2 
    \end{equation}
and the random variables
$\sigma\sim\operatorname{Unif}(\sqrt{N}\mathbb{S}^{N-1})$,
$Z\sim\texttt{Normal}(0,1)$, and $\mathrm{GOE}_{N-1}$
are mutually independent and independent of the disorder $J$.
\end{proposition}

\begin{proof}
    We ought to count the critical points via the Kac-Rice formula (see \cite[Theorem 11.2.1]{adler2009random}). The regularity conditions are automatically satisfied since we are only applying it to the Gaussian disorder. Let $d\sigma$ denote the
surface measure on $\sqrt{N}\mathbb{S}^{N-1}$. Then
    \begin{align*}
        \mathbb{E}_{J, G}\big[\text{Crt}_t(S)\big] &= \mathbb{E}_J\big[\mathbb{E}_G\big[\text{Crt}_t(S)\big]\big] \\
        &= \mathbb{E}_J\Big[\int_{\sqrt{N}\mathbb{S}^{N - 1}} \mathbb{E}_G\big[|\det \nabla_S^2 H_t(\sigma)| \cdot \mathbbm{1}_{\{H_t(\sigma) \in NS\}} \mid \nabla_S H_t(\sigma) = 0\big]f_{\nabla_SH_t(\sigma) \mid J}(0)\,d\sigma\Big],
    \end{align*}
   where $f_{\nabla_SH_t(\sigma)\mid J}$ denotes the conditional density
of $\nabla_SH_t(\sigma)$ given $J$. From this, together with the (conditional) covariance profile illustrated in Proposition \ref{Proposition:Conditional_Covariance}, we can conclude \eqref{Equation:Kac_Rice}.
\end{proof}

\section{Properties of the Hessian Matrix}\label{Section:Property}

\noindent  Throughout this subsection, assume $N\geq 2$. For convenience, let $g = \{g_{ij}\}_{1 \leq i \leq j \leq N - 1}$ be independent standard
Gaussian random variables, independent of $J$, and realize $\text{GOE}_{N - 1}$ by setting $(\text{GOE}_{N - 1})_{ij} = g_{ij}/\sqrt{N - 1}$ for $1 \leq i < j \leq N - 1$,
$(\text{GOE}_{N - 1})_{ii} = \sqrt{2}\,g_{ii}/\sqrt{N - 1}$, and $(\text{GOE}_{N - 1})_{ji} = (\text{GOE}_{N - 1})_{ij}$. We use the notation
\begin{equation}\label{Equation:N3}
    M_{N - 1}^{J, g}(\sigma) := M_{N - 1}^{J, \text{GOE}}(\sigma) \quad \text{and} \quad \mu_{N - 1}^{J, g, \sigma} := \mu_{M_{N - 1}^{J, g}(\sigma)}.
\end{equation}
The main purpose of this subsection is to study the concentration and the universality (in $J$) of the measure $\mu_{N - 1}^{J, g, \sigma}$.

\subsection{Concentration of the Hessian Matrix}\label{Subsection:U_1}

\begin{lemma}\label{Lemma:Lip}
    For any mapping $f \colon \mathbb{R} \to \mathbb{R}$ and $\sigma \in \sqrt{N}\mathbb{S}^{N - 1}$, we define $F_N^\sigma \colon \mathbb{R}^{N^p} \times \mathbb{R}^{N(N - 1)/2}  \to  \mathbb{R}$ by
    \[F_N^\sigma(J, g) = \int_{\mathbb{R}} f(\lambda) \,\mu_{N - 1}^{J, g, \sigma}(d\lambda) = \dfrac{1}{N - 1}\text{Tr}\big(f(M_{N - 1}^{J, g}(\sigma))\big).\]
    Then there exists a constant $\newconstant\label{Constant:Lip_Carry} = \oldconstant{Constant:Lip_Carry}(p) > 0$ such that for all $N \geq 2$,
    \begin{equation}\label{Equation:Lip_1}
        \sup_{\sigma}\|F_N^\sigma\|_{\text{Lip}} \leq \dfrac{\oldconstant{Constant:Lip_Carry}\|f\|_{\text{Lip}}}{N},
    \end{equation}
    where the supremum is over $\sigma\in \sqrt{N}\mathbb{S}^{N - 1}$
\end{lemma}

\begin{proof}
    For all pairs $(J, g), (J', g') \in \mathbb{R}^{N^p} \times \mathbb{R}^{N(N - 1)/2}$,
    \begin{align}
        |F_N^\sigma(J, g) - F_N^\sigma(J', g')| &= \dfrac{1}{N - 1}\bigg|\sum_{j = 1}^{N - 1} f\big(\lambda_j^{\uparrow}\big(M_{N - 1}^{J, g}(\sigma)\big)\big) - f\big(\lambda_j^{\uparrow}\big(M_{N - 1}^{J', g'}(\sigma)\big)\big)\bigg| \notag \\
        &\leq \dfrac{\|f\|_{\text{Lip}}}{N - 1}\sum_{j = 1}^{N - 1} \Big|\lambda_j^{\uparrow}\big(M_{N - 1}^{J, g}(\sigma)\big) - \lambda_j^{\uparrow}\big(M_{N - 1}^{J', g'}(\sigma)\big)\Big| \notag \\
        &\leq \dfrac{\|f\|_{\text{Lip}}}{\sqrt{N - 1}} \sqrt{\sum_{j = 1}^{N - 1} \Big|\lambda_j^{\uparrow}\big(M_{N - 1}^{J, g}(\sigma)\big) - \lambda_j^{\uparrow}\big(M_{N - 1}^{J', g'}(\sigma)\big)\Big|^2} \label{Equation:Concentration_2} \\
        &\leq \dfrac{\|f\|_{\text{Lip}}}{\sqrt{N - 1}}\Big\|M_{N - 1}^{J, g}(\sigma) - M_{N - 1}^{J', g'}(\sigma)\Big\|_{\text{F}}, \label{Equation:Concentration_3}
    \end{align}
    where \eqref{Equation:Concentration_2} holds by the Cauchy-Schwarz inequality and \eqref{Equation:Concentration_3} holds by the Hoffman-Wielandt inequality (see \cite[Lemma 2.1.19]{anderson2010introduction}). For the Hessian term in $M_{N - 1}^{J, g}(\sigma)$, we have
    \begin{align}
        &\Big\|\sqrt{t}\sqrt{\dfrac{N}{N - 1}}\dfrac{B(\sigma)^\top\nabla^2H_N^J(\sigma)B(\sigma)}{\sqrt{p(p - 1)}} - \sqrt{t}\sqrt{\dfrac{N}{N - 1}}\dfrac{B(\sigma)^\top\nabla^2H_N^{J'}(\sigma)B(\sigma)}{\sqrt{p(p - 1)}}\Big\|_{\text{F}}^2 \notag \\
        &= \dfrac{tN}{p(p - 1)(N - 1)} \Big\|B(\sigma)^\top\big(\nabla^2H_N^J(\sigma) - \nabla^2H_N^{J'}(\sigma)\big)B(\sigma)\Big\|_{\text{F}}^2 \notag \\
        &\leq \dfrac{N}{p(p - 1)(N - 1)}\Big\|\nabla^2H_N^J(\sigma) - \nabla^2H_N^{J'}(\sigma)\Big\|_{\text{F}}^2 \notag \\
        &\leq \dfrac{2}{p(p - 1)} \sum_{j, k = 1}^N\bigg[\dfrac{1}{N^{(p - 1)/2}}\sum_{I \in [N]^p} \big(J_I - J_I'\big)\dfrac{\partial^2 \sigma_I}{\partial \sigma_j\partial \sigma_k}\bigg]^2 \notag \\
        &\leq \dfrac{2}{p(p - 1)} \sum_{j, k = 1}^N \dfrac{1}{N^{p - 1}}\bigg[\sum_{I \in [N]^p} \big(J_I - J_I'\big)^2\mathbbm{1}_{\{j, k \in I\}} \bigg] \bigg[\sum_{I \in [N]^p}\Big(\dfrac{\partial^2 \sigma_I}{\partial \sigma_j\partial \sigma_k}\Big)^2 \bigg] \label{Equation:Concentration_4} \\
        &\leq \dfrac{2}{p(p - 1)} \sum_{j, k = 1}^N \dfrac{1}{N^{p - 1}}\bigg[\sum_{I \in [N]^p} \big(J_I - J_I'\big)^2\mathbbm{1}_{\{j, k \in I\}} \bigg]\bigg(p^2(p - 1)^2\sum_{I' \in [N]^{p - 2}} \sigma_{I'}^2\bigg) \label{Equation:Concentration_5}  \\
        &= \dfrac{2p(p - 1)}{N} \sum_{I \in [N]^p} \big(J_I - J_I'\big)^2\sum_{j, k = 1}^N\mathbbm{1}_{\{j, k \in I\}} \label{Equation:Concentration_6} \\
        &\leq \dfrac{2p^3(p - 1)}{N}\|J - J'\|_2^2, \label{Equation:Concentration_7}
    \end{align}
    where both \eqref{Equation:Concentration_4} and \eqref{Equation:Concentration_5} holds by the Cauchy-Schwarz inequality; and \eqref{Equation:Concentration_6} holds since $\|\sigma\|_2^2 = N$. For the GOE term in $M_{N - 1}^{J, g}(\sigma)$, we have
    \begin{equation}\label{Equation:Concentration_8}
        \big\|\sqrt{1 - t}\,\text{GOE}_{N - 1} - \sqrt{1 - t}\,\text{GOE}_{N - 1}'\big\|_{\text{F}} \leq \sqrt{\dfrac{2(1 - t)}{N - 1}\sum_{1 \leq i \leq j \leq N} (g_{ij} - g_{ij}')^2} \leq \dfrac{2\|g - g'\|_2}{\sqrt{N - 1}}.
    \end{equation}
    We see that \eqref{Equation:Lip_1} follows directly from \eqref{Equation:Concentration_3}, \eqref{Equation:Concentration_7}, and \eqref{Equation:Concentration_8}.
\end{proof}

For $\eta > 0$ and $K > 0$, define the function
\[\log_\eta(x) = \log|x + i\eta| = \dfrac{1}{2}\log(x^2 + \eta^2) \quad \text{and} \quad \log_\eta^K(x) = \min\{\log_\eta(x), \log_\eta(K)\}.\]
Note that $\log_\eta(x) \geq \log|x|$ and $\log_\eta(x) \geq \log_\eta^K(x)$ for all $x \in \mathbb{R}$. Moreover, both functions $\log_\eta(\cdot)$ and $\log_\eta^K(\cdot)$ are $(2\eta)^{-1}$-Lipschitz. Additionally, for all $x \in \mathbb{R}$, consider the random variable 
\[W_{\eta, N}^{J, g}(\sigma, x) = \int_{\mathbb{R}} \log_\eta(\lambda - x)\,\mu_{N - 1}^{J, g, \sigma}(d\lambda) = \dfrac{1}{N - 1}\text{Tr}\big[\log_\eta\big(M_{N - 1}^{J, g}(\sigma) - xI_{N - 1}\big)\big].\]

\begin{lemma}\label{Lemma:Lip_Concentration}
    Suppose $\{J_I\}_{I \in [N]^p, N \in \mathbb{N}}$ satisfies \eqref{Equation:SG} with uniform $\kappa_{\text{SG}}$, then there exists a constant $\newconstant\label{Constant:Concentration} = \oldconstant{Constant:Concentration}(\kappa_{\text{SG}}, p) > 0$ such that for any $h \geq 0$,
    \begin{equation}\label{Equation:Concentration_Hessian_1}
        \sup_{\sigma, x}\mathbb{P}_{J, g}\big(\big|W_{\eta, N}^{J, g}(\sigma, x) - \mathbb{E}_{J, g}\big[W_{\eta, N}^{J, g}(\sigma, x)\big]\big| \geq h \big) \leq \oldconstant{Constant:Concentration}^{-1}\exp\Big(-\dfrac{\oldconstant{Constant:Concentration}\eta^2N^2h^2}{\log N}\Big)
    \end{equation}
    and for any $N\geq 2$ and $\lambda > 0$,
    \begin{equation}\label{Equation:Concentration_Hessian_2}
        \sup_{\sigma, x}\mathbb{E}_{J, g}\Big[\exp\Big(\lambda \big(W_{\eta, N}^{J, g}(\sigma, x) - \mathbb{E}_{J, g}\big[W_{\eta, N}^{J, g}(\sigma, x)\big]\big)\Big)\Big] \leq \oldconstant{Constant:Concentration}^{-1}\exp\Big(\frac{\lambda^2\log N}{\oldconstant{Constant:Concentration}\eta^2N^2}\Big),
    \end{equation}
    where the supremum is over $\sigma\in \sqrt{N}\mathbb{S}^{N - 1}$ and $x \in \mathbb{R}$.
\end{lemma}

\begin{proof}
    \eqref{Equation:Concentration_Hessian_2} follows from \eqref{Equation:Concentration_Hessian_1} directly by enlarging the constant \oldconstant{Constant:Concentration} if necessary. 
    We introduce the following auxiliary functions:
    \begin{align*}
        \log_\eta^1(\lambda) &= \begin{cases}
            -\frac{\lambda}{2\eta} - \frac{1}{2} + \log_\eta(\eta), & \lambda \leq -\eta, \\
            \log_\eta(\lambda), & -\eta \leq \lambda \leq \eta, \\
            \frac{\lambda}{2\eta} - \frac{1}{2} + \log_\eta(\eta), & \lambda \geq \eta,
        \end{cases} \\
        \log_\eta^2(\lambda) &= \begin{cases}
            \log_\eta(\lambda) + \frac{1}{2} - \log_\eta(\eta), & \lambda \leq -\eta, \\
            -\frac{\lambda}{2\eta}, & \lambda \geq -\eta,
        \end{cases} \\
        \log_\eta^3(\lambda) &= \begin{cases}
            \frac{\lambda}{2\eta}, & \lambda \leq \eta,\\
            \log_\eta(\lambda) + \frac{1}{2} - \log_\eta(\eta), & \lambda \geq \eta.
        \end{cases}
    \end{align*}
    By definition, $\log_\eta(\lambda) = \log_\eta^1(\lambda) + \log_\eta^2(\lambda) + \log_\eta^3(\lambda)$ for all $\lambda \in \mathbb{R}$ and $\eta > 0$; $\log_\eta^1(\lambda)$, $-\log_\eta^2(\lambda)$, and $-\log_\eta^3(\lambda)$ are convex; and $\|\log_\eta^1(\cdot)\|_{\text{Lip}}$, $\|\log_\eta^2(\cdot)\|_{\text{Lip}}$, $\|\log_\eta^3(\cdot)\|_{\text{Lip}} \leq (2\eta)^{-1}$.
     For $i = 1, 2, 3$, set
    \[W_{i, \eta, N}^{J, g}(\sigma, x) = (-1)^{\mathbbm{1}_{\{i \neq 1\}}}\int_{\mathbb{R}} \log_\eta^i(\lambda - x)\,\mu_{N - 1}^{J, g, \sigma}(d\lambda).\]
    Since the mapping $(J, g) \mapsto M_{N - 1}^{J, g}(\sigma)$ is affine and $\lambda \mapsto (-1)^{\mathbbm{1}_{\{i \neq 1\}}}\log_\eta^i(\lambda - x)$ is convex with Lipschitz constant $\leq (2\eta)^{-1}$ (for $i = 1, 2, 3$), we have from Klein's lemma \cite[Lemma 1.2]{guionnet2000concentration} and Lemma \ref{Lemma:Lip} that $(J, g) \mapsto W_{i, \eta, N}^{J, g}(\sigma, x)$ is convex and $(\oldconstant{Constant:Lip_Carry}/2\eta N)$-Lipschitz (uniformly over deterministic realization of $g \in \mathbb{R}^{N(N - 1)/2}$, $\sigma \in \sqrt{N}\mathbb{S}^{N - 1}$, and $x \in \mathbb{R}$). Therefore, from \cite[Proposition A.1]{cook2023full}, there exists an absolute constant $c > 0$ such that
    \begin{align}
        \mathbb{P}_{J, g}\big(\big|W_{\eta, N}^{J, g}(\sigma, x) - \mathbb{E}_{J, g}\big[W_{\eta, N}^{J, g}(\sigma, x)\big] \big| \geq h\big)
        &\leq \sum_{i = 1}^3 \mathbb{P}_{J, g}\Big(\big|W_{i, \eta, N}^{J, g}(\sigma, x) - \mathbb{E}_{J, g}\big[W_{i, \eta, N}^{J, g}(\sigma, x)\big]\big| \geq \dfrac{h}{3}\Big) \notag \\
        &\leq 6\exp\Big(-\dfrac{4c\eta^2N^2h^2}{9\oldconstant{Constant:Lip_Carry}^2\max\{\kappa_{\text{SG}}, 2\}^2\log(N^p + N^2)}\Big) \label{Equation:Concentration_Hessian_3}
    \end{align}
    for all $h \geq 0$. The desired inequality \eqref{Equation:Concentration_Hessian_1} then follows by \eqref{Equation:Concentration_Hessian_3}.
\end{proof}

\begin{lemma}\label{Proposition:Op_Bound}
    Suppose $\{J_I\}_{I \in [N]^p, N \in \mathbb{N}}$ satisfies \eqref{Equation:SG} with uniform $\kappa_{\text{SG}}$. Then, the following statements hold uniformly over $\sigma \in \sqrt{N}\mathbb{S}^{N - 1}$.
    \begin{enumerate}[label = (\roman*)]
        \item Recall $x_N^{J, Z}(\sigma)$ from \eqref{Equation:Def_x}. There exists   $\newconstant\label{Constant:SG_H} = \oldconstant{Constant:SG_H}(\kappa_{\text{SG}}, p) > 0$ such that
        \begin{equation}\label{Equation:Concentration_H}
            \sup_\sigma \mathbb{P}_{J, Z}\big(|x_N^{J, Z}(\sigma)| \geq h\big) \leq 2e^{-\oldconstant{Constant:SG_H}Nh^2} \quad \text{for all} \quad h > 0.
        \end{equation}
        \item Recall $M_{N - 1}^{J, g}(\sigma)$ from \eqref{Equation:Def_M}. For all $C > 0$, there exists $\newconstant\label{Constant:Op_Bound} = \oldconstant{Constant:Op_Bound}(\kappa_{\text{SG}}, p, C) > 0$ such that
        \begin{equation}\label{Equation:Op_Bound}
            \sup_{N \geq 2}\sup_{\sigma, k \leq CN}\mathbb{E}_{J, g}\big[\|M_{N - 1}^{J, g}(\sigma)\|_{\text{op}}^k\big]^{1/k} \leq \oldconstant{Constant:Op_Bound}.
        \end{equation}
        \item There exists $\newconstant\label{Constant:Det_Bound} = \oldconstant{Constant:Det_Bound}(\kappa_{\text{SG}}, p) > 0$ such that
        \begin{equation}\label{Equation:Det_Bound}
            \sup_{N \geq 2}\sup_{\sigma} \mathbb{E}_{J, g, Z}\big[\big|\det\big(M_{N - 1}^{J, g}(\sigma) - x_N^{J, Z}(\sigma)I_{N - 1}\big)\big|^2\big]^{1/2} \leq e^{\oldconstant{Constant:Det_Bound}N}.
        \end{equation}
    \end{enumerate}
\end{lemma}

\begin{proof}
    \eqref{Equation:Concentration_H} holds directly from Markov's inequality and \eqref{Equation:SG}:
    \begin{align*}
        \mathbb{P}_{J, Z}\big(x_N^{J, Z}(\sigma) \geq h\big) &\leq e^{-\lambda h}\mathbb{E}_{J, Z}\big[e^{\lambda x_N^{J, Z}(\sigma)}\big] \\
        &= e^{-\lambda h}\mathbb{E}_{J, Z}\Big[\exp\Big(\lambda \sqrt{\dfrac{p}{p - 1}}\Big(\sqrt{\dfrac{t}{N(N - 1)}}\sum_{I \in [N]^p} \dfrac{\sigma_I}{N^{(p - 1)/2}}J_I + \sqrt{\dfrac{1 - t}{N - 1}}Z\Big)\Big)\Big] \\
        &\leq \exp\Big[-\lambda h + \dfrac{\max\{\kappa_{\text{SG}}^2, 1\}\lambda^2p}{2(p - 1)}\Big(\dfrac{1}{N(N - 1)}\sum_{I \in [N]^p} \dfrac{\sigma_I^2}{N^{p - 1}} + \dfrac{1}{N - 1}\Big)\Big] \\
        &= \exp\Big(-\lambda h + \dfrac{\max\{\kappa_{\text{SG}}^2, 1\}\lambda^2p}{(p - 1)(N - 1)}\Big)
    \end{align*}
    Optimizing over $\lambda \geq 0$ gives us a one-sided bound in the form \eqref{Equation:Concentration_H}. By applying the same argument on $-x_N^{J, Z}(\sigma)$, we get the two-sided bound \eqref{Equation:Concentration_H}.
    
    In order to show \eqref{Equation:Op_Bound}, we see by definition \eqref{Equation:Def_M} and the tail probability estimate \eqref{Equation:Tail_of_GOE} that it suffices to prove $\mathbb{E}_J\big[\|\nabla^2H_N^J(\sigma)\|_{\text{op}}^k\big]^{1/k}$ can be bounded uniformly as long as $k \leq CN$. Observe that for all $x \in \mathbb{S}^{N - 1}$,
    \[\langle \nabla^2H_N^J(\sigma)x, x\rangle = \sum_{I \in [N]^p} \Big(\dfrac{1}{N^{(p - 1)/2}}\sum_{j, k = 1}^N\dfrac{\partial^2\sigma_I}{\partial \sigma_j\partial \sigma_k}x_jx_k\Big)J_I.\]
    Note that, as in the proof of Lemma \ref{Lemma:Lip}, the $2$-norm of the coefficients satisfies
    \[\sum_{I \in [N]^p} \Big(\dfrac{1}{N^{(p - 1)/2}}\sum_{j, k = 1}^N\dfrac{\partial^2\sigma_I}{\partial \sigma_j\partial \sigma_k}x_jx_k\Big)^2 \leq \dfrac{p^2(p - 1)^2}{N}.\]
    It follows that $\langle \nabla^2H_N^J(\sigma) x, x\rangle$ is $(p^2(p - 1)^2\kappa_{\text{SG}}^2/N)$-subgaussian by \cite[Proposition 2.6.1]{Vershynin2026}. 

    Now, for $0<\delta<1/2$, let $\mathcal{N}_\delta$ be a $\delta$-net of the unit sphere $\mathbb{S}^{N - 1}$. By \cite[Proposition 4.2.12]{Vershynin2026}, we may choose it so that $|\mathcal{N}_\delta|\leq (C_0/\delta)^N$ for an absolute constant $C_0 > 0$. For every $x\in\mathbb{S}^{N-1}$, there exists $y\in\mathcal{N}_\delta$ such that $\|x-y\|_2\leq\delta$. It follows that
    \begin{equation}\label{Equation:Hessian_op}
        \begin{aligned}
            \|\nabla^2H_N^J(\sigma)\|_{\text{op}} &= \max_{x\in\mathbb{S}^{N-1}} \big|\langle \nabla^2H_N^J(\sigma)x,x\rangle\big| \\
            &\leq \max_{y\in\mathcal{N}_\delta} \big|\langle \nabla^2H_N^J(\sigma)y, y\rangle\big| + 2\delta\|\nabla^2H_N^J(\sigma)\|_{\text{op}}.
        \end{aligned}
    \end{equation}
    Apply \eqref{Equation:Sup_SG} to the $2|\mathcal{N}_\delta|$ random variables $\big\{\pm \langle \nabla^2H_N^J(\sigma)y, y\rangle \colon y \in \mathcal{N}_\delta\big\}$. Together with \eqref{Equation:Hessian_op}, this yields
    \begin{align*}
        \mathbb{E}_J\big[\|\nabla^2H_N^J(\sigma)\|_{\text{op}}^k\big]^{1/k} &\leq \dfrac{1}{1 - 2\delta}\mathbb{E}_J\Big[\Big(\max_{y \in \mathcal{N}_\delta} \langle \nabla^2H_N^J(\sigma)y, y\rangle\Big)^k\Big]^{1/k} \\
        &\leq \dfrac{\oldconstant{Constant:Sup_SG}}{1 - 2\delta} \cdot \dfrac{p(p - 1)\kappa_{\text{SG}}}{\sqrt{N}}\sqrt{\log ( 2|\mathcal{N}_\delta|) + k} = O_N(1)
    \end{align*}
    for fixed $\delta$ and $k\leq CN$. This proves \eqref{Equation:Op_Bound}.
    
    Last, \eqref{Equation:Det_Bound} follows from
    \[\mathbb{E}_{J, g, Z}\big[\big|\det\big(M_{N - 1}^{J, g}(\sigma) - x_N^{J, Z}(\sigma)I_{N - 1}\big)\big|^2\big] \leq \mathbb{E}_{J, g, Z} \Big[\big(\| M_{N - 1}^{J, g}(\sigma)\|_{\text{op}} + |x_N^{J, Z}(\sigma)|\big)^{2(N - 1)}\Big],\]
    the inequality \eqref{Equation:Concentration_H}, and \eqref{Equation:Op_Bound} with $k = 2(N - 1)$.
\end{proof}

\begin{corollary}
    If $\{J_I\}_{I \in [N]^p, N \in \mathbb{N}}$ satisfies \eqref{Equation:SG} with uniform $\kappa_{\text{SG}}$, then there exists a constant $\newconstant\label{Constant:Upper_Truncation} = \oldconstant{Constant:Upper_Truncation}(\kappa_{\text{SG}}, p) > 0$ such that for all $M \geq 1$, we have
    \begin{equation}\label{Equation:Upper_Truncation}
        \sup_{N \ge 2}\sup_{\sigma, |x| \leq M, \eta > 0}\int_{\mathbb{R}} \big(\log_\eta(\lambda - x) - \log_\eta^K(\lambda - x)\big)\,\mathbb{E}_{J, g}\big[\mu_{N - 1}^{J, g, \sigma}\big](d\lambda) \leq \dfrac{\oldconstant{Constant:Upper_Truncation}M^2}{K^2},
    \end{equation}
    where the supremum is over $\sigma \in \sqrt{N}\mathbb{S}^{N - 1}$.
\end{corollary}

\begin{proof}
    By definition, rewrite
    \begin{align}
        &\int_{\mathbb{R}} \big(\log_\eta(\lambda - x) - \log_\eta^K(\lambda - x)\big)\,\mathbb{E}_{J, g}\big[\mu_{N - 1}^{J, g, \sigma}\big](d\lambda) \notag \\
        &= \dfrac{1}{2}\int_{|\lambda - x| \geq K} \log\Big(\dfrac{(\lambda - x)^2 + \eta^2}{K^2 + \eta^2}\Big)\,\mathbb{E}_{J, g}\big[\mu_{N - 1}^{J, g, \sigma}\big](d\lambda) \notag \\
        &= \dfrac{1}{2}\int_K^\infty \Big(\dfrac{2y}{y^2 + \eta^2}\Big)\mathbb{E}_{J, g}\big[\mu_{N - 1}^{J, g, \sigma}\big](|\lambda - x| \geq y)\,dy \notag \\
        &\leq \int_{\mathbb{R}} |\lambda - x|^2\,\mathbb{E}_{J, g}\big[\mu_{N - 1}^{J, g, \sigma}\big](d\lambda) \int_{K}^\infty \dfrac{2}{y(y^2 + \eta^2)}\,dy, \label{Equation:Upper_Truncation_1}
    \end{align}
    where \eqref{Equation:Upper_Truncation_1} holds by Markov's inequality. By \eqref{Equation:Op_Bound},  the first term in \eqref{Equation:Upper_Truncation_1} satisfies
    \[\int_{\mathbb{R}} |\lambda - x|^2\,\mathbb{E}_{J, g}\big[\mu_{N - 1}^{J, g, \sigma}\big](d\lambda) \leq \mathbb{E}_{J, g}\big[\|M_{N - 1}^{J, g}(\sigma)\|_{\text{op}}^2 + 2|x|\|M_{N - 1}^{J, g}(\sigma)\|_{\text{op}} + x^2\big] = M^2O_N(1).\]
    Moreover, the second term in \eqref{Equation:Upper_Truncation_1} is of order $K^{-2}$, uniformly over all $\eta > 0$. These two estimations give us the desired bound \eqref{Equation:Upper_Truncation}.
\end{proof}

\subsection{Universality of the Hessian Matrix}\label{Subsection:U_2}

\noindent Throughout this subsection, we assume that the disorder variables
$J_I$'s are independent and have zero mean and unit variance.

\begin{definition}[Finite Moment] 
    A collection of random variables $\{X_\alpha\}_{\alpha \in A}$ is said to have uniformly bounded moments if for all $k \in \mathbb{N}$
    \begin{equation}\label{Equation:M}
        \sup_{\alpha \in A} \mathbb{E}\big[|X_\alpha|^k\big] < \infty. \tag{M}
    \end{equation}
\end{definition}

We note that condition \eqref{Equation:M} holds if the same collection of random variables satisfies \eqref{Equation:SG} with uniform $\kappa_{\text{SG}}$.

\begin{lemma}
    If $\{J_I\}_{I \in [N]^p, N \in \mathbb{N}}$ satisfies \eqref{Equation:M}, then
    \begin{equation}\label{Equation:Universal_Hessian_W1}
        \lim_{N \to \infty} \sup_{\sigma} d_{\text{BL}}\big(\mathbb{E}_{J, g}\big[\mu_{N - 1}^{J, g, \sigma}\big], \mu_{\text{sc}}\big) = 0,
    \end{equation}
    where the supremum is over $\sigma \in \sqrt{N}\mathbb{S}^{N - 1}$.
\end{lemma}

\begin{proof}
    Conditioning on $\sigma \in \sqrt{N}\mathbb{S}^{N - 1}$, $\mu_{N - 1}^{J, g, \sigma}$ is the empirical measure of the random matrix
    \[M_{N - 1}^{J, g}(\sigma) \overset{d}{=} \sqrt{\dfrac{N}{N - 1}}\dfrac{B(\sigma)^\top\nabla^2 H_N^{\sqrt{1 - t}G + \sqrt{t}J}(\sigma)B(\sigma)}{\sqrt{p(p - 1)}}.\]
    Recall the rank $(N - 1)$ projection $P(\sigma)$ from \eqref{Equation:Projection_2} and we set
    \[\overline{M}_{N - 1}^{J, g}(\sigma) = \dfrac{P(\sigma)^\top\nabla^2 H_N^{\sqrt{1 - t}G + \sqrt{t}J}(\sigma)P(\sigma)}{\sqrt{p(p - 1)}}.\]
    Then, the bounded Lipschitz distance satisfies
    \begin{align}
        d_{\text{BL}}\big(\mathbb{E}_{J, g}\big[\mu_{N - 1}^{J, g, \sigma}\big], \mu_{\text{sc}}\big) &\leq d_{\text{BL}}\Big(\mathbb{E}_{J, g}\Big[\mu_{M_{N - 1}^{J, g}(\sigma)}\Big], \mathbb{E}_{J, g}\Big[\mu_{\sqrt{\frac{N - 1}{N}}M_{N - 1}^{J, g}(\sigma)}\Big]\Big) \label{Equation:Universal_1} \\
        &\hspace{1cm} + d_{\text{BL}}\Big(\mathbb{E}_{J, g}\Big[\mu_{\sqrt{\frac{N - 1}{N}}M_{N - 1}^{J, g}(\sigma)}\Big], \mathbb{E}_{J, g}\Big[\mu_{\overline{M}_{N - 1}^{J, g}(\sigma)}\Big]\Big) \label{Equation:Universal_2} \\
        &\hspace{1cm} + d_{\text{BL}}\Big(\mathbb{E}_{J, g}\Big[\mu_{\overline{M}_{N - 1}^{J, g}(\sigma)}\Big], \mu_{\text{sc}}\Big). \label{Equation:Universal_3}
    \end{align}
    By \eqref{Equation:Op_Bound}, the first term \eqref{Equation:Universal_1} can be estimated by
    \[d_{\text{BL}}\Big(\mathbb{E}_{J, g}\Big[\mu_{M_{N - 1}^{J, g}(\sigma)}\Big], \mathbb{E}_{J, g}\Big[\mu_{\sqrt{\frac{N - 1}{N}}M_{N - 1}^{J, g}(\sigma)}\Big]\Big) \leq \Big[\Big(\dfrac{N}{N - 1}\Big)^{1/2} - 1\Big]\mathbb{E}_{J, g}\big[\|M_{N - 1}^{J, g}(\sigma)\|_{\text{op}}\big] = o_N(1),\]
    where the convergence is uniform over $\sigma \in \sqrt{N}\mathbb{S}^{N - 1}$. By definition of $B(\sigma)$ and $P(\sigma)$, the second term \eqref{Equation:Universal_2} can be estimated by
    \begin{align*}
        &d_{\text{BL}}\Big(\mathbb{E}_{J, g}\Big[\mu_{\sqrt{\frac{N - 1}{N}}M_{N - 1}^{J, g}(\sigma)}\Big], \mathbb{E}_{J, g}\Big[\mu_{\overline{M}_{N - 1}^{J, g}(\sigma)}\Big]\Big) \\
        &= d_{\text{BL}}\Big(\mathbb{E}_{J, g}\Big[\mu_{\sqrt{\frac{N - 1}{N}}M_{N - 1}^{J, g}(\sigma)}\Big], \frac{1}{N}\delta_0 + \frac{N - 1}{N}\mathbb{E}_{J, g}\Big[\mu_{\sqrt{\frac{N - 1}{N}}M_{N - 1}^{J, g}(\sigma)}\Big]\Big) \leq \dfrac{2}{N}.
    \end{align*}
    Last, note that the interpolated disorder $\sqrt{1 - t}\,G_I + \sqrt{t}J_I$ still satisfies \eqref{Equation:M}. By \cite[Proposition 3.1]{ho2025hessian}, the term in \eqref{Equation:Universal_3} also converge uniformly to $0$ as $N \to \infty$. This concludes the proof of the lemma.
\end{proof}

\begin{lemma}\label{Lemma:Truncation_of_Log}
    For all $h > 0$, there exists $\eta = \eta(h) > 0$ such that
    \begin{equation}\label{Equation:Truncation_Log}
        \sup_{N, A_N, x}\int_{\mathbb{R}} \big(\log_{\eta}(\lambda - x) - \log|\lambda - x|\big)\,\mathbb{E}[\mu_{A_N + \text{GOE}_N}](d\lambda) \leq h,
    \end{equation}
    where the supremum is over all $N \in \mathbb{N}$, deterministic symmetric matrices $A_N$, and $x \in \mathbb{R}$.
\end{lemma}

\begin{proof}
    From \cite[Theorem 1]{aizenman2017matrix}, there exists an absolute constant $\newconstant\label{Constant:General_Wegner}>0$ such that, for every $N \in \mathbb{N}$, every deterministic symmetric $N \times N$ matrix $A_N$, and every interval $I\subseteq\mathbb{R}$, we have
    \begin{equation}\label{Equation:General_Wegner}
        \mathbb{E}(|\text{Spec}(A_N + \text{GOE}_N) \cap I|) \leq \oldconstant{Constant:General_Wegner}N|I|,
    \end{equation}
    where $|I|$ is the Lebesgue measure of $I$. \eqref{Equation:General_Wegner} implies that the empirical measure satisfies
    \[\mathbb{E}\left[\mu_{A_N + \text{GOE}_N}\right](I) = \dfrac{1}{N}\mathbb{E}(|\text{Spec}(A_N + \text{GOE}_N) \cap I|) \leq \oldconstant{Constant:General_Wegner}|I|\]
    for every interval $I$. Therefore, $\mathbb{E}\left[\mu_{A_N + \text{GOE}_N}\right]$ is absolutely continuous with respect to the Lebesgue measure with density bounded by $\oldconstant{Constant:General_Wegner}$. Consequently,
    \begin{align*}
        \int_{\mathbb{R}} \big(\log_{\eta}(\lambda - x) - \log|\lambda - x|\big)\,\mathbb{E}[\mu_{A_N + \text{GOE}_N}](d\lambda) &\leq \dfrac{\oldconstant{Constant:General_Wegner}}{2}\int_{\mathbb{R}} \log\Big(1 + \dfrac{\eta^2}{(\lambda - x)^2}\Big)\,d\lambda = \oldconstant{Constant:General_Wegner}\pi\eta.
    \end{align*}
    We see then \eqref{Equation:Truncation_Log} holds directly.
\end{proof}

\section{Isolating the  Determinant}\label{Section:Isolation}

\noindent Recall $E_\infty = 2\sqrt{(p - 1)/p}$ from \eqref{Equation:Energy}. For all $u \in \mathbb{R}$, let
\[U_N(u) := \sqrt{\frac{Np}{(N - 1)(p - 1)}}(-\infty, u) = \Big(-\infty, \sqrt{\dfrac{N}{N - 1}}\dfrac{2u}{E_\infty}\Big).\]
The main goal of this section is to isolate the determinant term in
\begin{equation}\label{Equation:L_N(u)}
  \mathcal{L}_N(u) := \mathbb{E}_{\sigma, J, g, Z}\Big[e^{-\frac{Nt}{2p(1 - t)}y_N^J(\sigma)}\big|\det \big(M_{N - 1}^{J, g}(\sigma) - x_N^{J, Z}(\sigma)I_{N - 1}\big)\big| \cdot \mathbbm{1}_{\{x_N^{J, Z} \in U_N(u)\}}\Big]  
\end{equation}
by using the concentration of the random matrix in the determinant. For all $x \in \mathbb{R}$ and $\delta > 0$, we also introduce
\begin{equation}\label{Equation:ell_N}
 \ell_N(x, \delta) := \mathbb{E}_{\sigma, J, Z}\Big[e^{-\frac{Nt}{2p(1 - t)}y_N^J(\sigma)} \mathbbm{1}_{\{x_N^{J, Z}(\sigma) \in B(x, \delta)\}}\Big].   
\end{equation}

\begin{lemma}
    For $K \geq M + 2$ and $|x| \leq M$,
    \begin{equation}\label{Equation:Log_Bound}
        0 \leq \int_{\mathbb{R}}\big(\log_\eta^K(\lambda - x) - \log|\lambda - x|\big)\,\mu_{\text{sc}}(d\lambda) \leq \eta.
    \end{equation}
\end{lemma}

\begin{proof}
    Since $\text{supp}(\mu_{\text{sc}}) \subseteq [-2, 2]$, we must have $|\lambda - x| \leq M + 2$. Therefore,
    \begin{align}
        &\int_{\mathbb{R}} \big(\log_\eta^K(\lambda - x) - \log|\lambda - x|\big)\,\mu_{\text{sc}}(d\lambda) \notag \\
        &= \int_{\mathbb{R}} \big(\log_\eta(\lambda - x) - \log|\lambda - x|\big)\,\mu_{\text{sc}}(d\lambda) - \int_{|\lambda - x| \geq K} (\log_\eta(\lambda - x) - \log_\eta^K(\lambda - x))\,\mu_{\text{sc}}(d\lambda) \label{Equation:Log_Bound_1} \\
        &= \dfrac{1}{2}\int_{\mathbb{R}} \log\Big(1 + \dfrac{\eta^2}{(\lambda - x)^2}\Big) \,\mu_{\text{sc}}(d\lambda) \label{Equation:Log_Bound_2} \\
        &\leq \dfrac{1}{2\pi} \int_{\mathbb{R}} \log\Big(1 + \dfrac{\eta^2}{(\lambda - x)^2}\Big) \,d\lambda = \eta, \label{Equation:Log_Bound_3}
    \end{align}
    where \eqref{Equation:Log_Bound_2} holds since the second integral in \eqref{Equation:Log_Bound_1} vanishes and \eqref{Equation:Log_Bound_3} by the density \eqref{Equation:Energy}. The lower bound in \eqref{Equation:Log_Bound} is direct from the expression \eqref{Equation:Log_Bound_2}.
\end{proof}

\subsection{Upper Bound}

\begin{proposition}\label{Theorem:Upper}
    If $\{J_I\}_{I \in [N]^p, N \in \mathbb{N}}$ satisfies \eqref{Equation:SG} with uniform $\kappa_{\text{SG}}$, then, for $u \in \mathbb{R}$, 
    \[\limsup_{N \to \infty}\dfrac{1}{N}\log \mathcal{L}_N(u) \leq \lim_{\delta \to 0} \sup_{x \leq 2u/E_\infty} \Big\{\Omega(x) + \limsup_{N \to \infty} \dfrac{1}{N}\log \ell_N(x, \delta)\Big\}.\]
\end{proposition}

\begin{proof}
    Let $\delta, \eta, h > 0$ and let $M \geq 2 + |2u/E_\infty|$, $K = M^2 \geq M + 2$ be parameters to be chosen sufficiently large. By compactness, there exists finitely many points $x_1, \dots, x_n \in [-M, 2u/E_\infty]$, including $x_n = 2u/E_\infty$, such that $[-M,2u/E_\infty] \subseteq \bigcup_{i=1}^n B(x_i,\delta)$. Consequently, $U_N(u) \cap [-M,M] \subseteq \bigcup_{i=1}^n B(x_i, \delta)$ for all sufficiently large $N$.
    
    Therefore, for all sufficiently large $N$, 
    \begin{align}
        &\mathbb{E}_{J, g, Z}\Big[e^{-\frac{Nt}{2p(1 - t)}y_N^J(\sigma)}\big|\det \big(M_{N - 1}^{J, g}(\sigma) - x_N^{J, Z}(\sigma)I_{N - 1}\big)\big| \cdot \mathbbm{1}_{\{x_N^{J, Z}(\sigma) \in U_N(u)\}}\Big] \notag \\
        &= \mathbb{E}_{J, g, Z}\Big[e^{-\frac{Nt}{2p(1 - t)}y_N^J(\sigma)} \cdot e^{(N - 1)\int_{\mathbb{R}} \log |\lambda - x_N^{J, Z}(\sigma)|\,\mu_{N - 1}^{J, g, \sigma}(d\lambda)} \mathbbm{1}_{\{x_N^{J, Z}(\sigma) \in U_N(u)\}}\Big] \notag \\
        &\leq \mathbb{E}_{J, g, Z}\Big[e^{(N - 1)\int_{\mathbb{R}} \log |\lambda - x_N^{J, Z}(\sigma)|\,\mu_{N - 1}^{J, g, \sigma}(d\lambda)}\mathbbm{1}_{\{|x_N^{J, Z}(\sigma)| \geq M\}}\Big] \label{Equation:Upper_1} \\
        &\hspace{1cm} + \sum_{i = 1}^n \mathbb{E}_{J, g, Z}\Big[e^{-\frac{Nt}{2p(1 - t)}y_N^J(\sigma)} \cdot e^{(N - 1)\int_{\mathbb{R}} \log |\lambda - x_N^{J, Z}(\sigma)|\,\mu_{N - 1}^{J, g, \sigma}(d\lambda)}\mathbbm{1}_{\{x_N^{J, Z}(\sigma) \in B(x_i, \delta)\}}\Big]. \label{Equation:Upper_2}
    \end{align}
    For the first term \eqref{Equation:Upper_1}, we see that there exists a constant $\newconstant\label{Constant:Upper_1} = \oldconstant{Constant:Upper_1}(\kappa_{\text{SG}}, p) > 0$ such that
    \begin{align}
        &\mathbb{E}_{J, g, Z}\Big[e^{(N - 1)\int_{\mathbb{R}} \log |\lambda - x_N^{J, Z}(\sigma)|\,\mu_{N - 1}^{J, g, \sigma}(d\lambda)}\mathbbm{1}_{\{|x_N^{J, Z}(\sigma)| \geq M\}}\Big] \notag \\
        &\leq \mathbb{E}_{J, g, Z}\Big[e^{2(N - 1)\int_{\mathbb{R}} \log |\lambda - x_N^{J, Z}(\sigma)|\,\mu_{N - 1}^{J, g, \sigma}(d\lambda)}\Big]^{1/2} \cdot \mathbb{P}_{J, Z}\big(|x_N^{J, Z}(\sigma)| \geq M\big)^{1/2} \notag \\
        &= \mathbb{E}_{J, g, Z}\big[\big|\det\big(M_{N - 1}^{J, g}(\sigma) - x_N^{J, Z}(\sigma)I_{N - 1}\big)\big|^2\big]^{1/2} \cdot \mathbb{P}_{J, Z}\big(|x_N^{J, Z}(\sigma)| \geq M\big)^{1/2} \notag \\
        &\leq e^{\oldconstant{Constant:Det_Bound}N} \cdot \sqrt{2}e^{-\oldconstant{Constant:SG_H}NM^2/2} \label{Equation:Upper_3} \\
        &\leq e^{-\oldconstant{Constant:Upper_1}NM^2}, \label{Equation:Upper_4}
    \end{align}
    for sufficiently large $M$, where the \eqref{Equation:Upper_3} follows from \eqref{Equation:Concentration_H} and \eqref{Equation:Det_Bound}. Next, let
    \begin{align*}
        A_i &= \int_{\mathbb{R}} \big(\log_\eta (\lambda - x_N^{J, Z}(\sigma)) - \log_\eta(\lambda - x_i)\big)\,\mu_{N - 1}^{J, g, \sigma}(d\lambda), \\
        B_i &= \int_{\mathbb{R}} \log_\eta(\lambda - x_i)\big(\mu_{N - 1}^{J, g, \sigma} - \mathbb{E}_{J, g}\big[\mu_{N - 1}^{J, g, \sigma}\big]\big)(d\lambda), \\
        C_i &= \int_{\mathbb{R}} \big(\log_\eta(\lambda - x_i) - \log_\eta^K(\lambda - x_i)\big)\,\mathbb{E}_{J, g}\big[\mu_{N - 1}^{J, g, \sigma}\big](d\lambda), \\
        D_i &= \int_{\mathbb{R}} \log_\eta^K(\lambda - x_i)\,\big(\mathbb{E}_{J, g}\big[\mu_{N - 1}^{J, g, \sigma}\big] - \mu_{\text{sc}}\big)(d\lambda) \\
        E_i &= \int_{\mathbb{R}} (\log_\eta^K(\lambda - x_i) - \log|\lambda - x_i|)\,\mu_{\text{sc}}(d\lambda).
    \end{align*}
    Then
    \begin{align}
       \int_{\mathbb{R}} \log |\lambda - x_N^{J, Z}(\sigma)|\,\mu_{N - 1}^{J, g, \sigma}(d\lambda) &\le \int_{\mathbb{R}} \log_\eta \big( \lambda - x_N^{J, Z}(\sigma) \big)\,\mu_{N - 1}^{J, g, \sigma}(d\lambda) \notag \\
        &= A_i + B_i + C_i + D_i + E_i + \Omega(x_i). \label{Equation:Upper_5}
    \end{align}
    From \eqref{Equation:Upper_1}, \eqref{Equation:Upper_2}, \eqref{Equation:Upper_4}, and \eqref{Equation:Upper_5}, it follows that
    \begin{align}
        &\mathbb{E}_{J, g, Z}\Big[e^{-\frac{Nt}{2p(1 - t)}y_N^J(\sigma)}\big|\det \big(M_{N - 1}^{J, g}(\sigma) - x_N^{J, Z}(\sigma)I_{N - 1}\big)\big| \cdot \mathbbm{1}_{\{x_N^{J, Z}(\sigma) \in U_N(u)\}}\Big] \notag \\
        &\leq e^{-\oldconstant{Constant:Upper_1}NM^2} + \sum_{i = 1}^n \mathbb{E}_{J, g, Z}\Big [e^{-\frac{Nt}{2p(1 - t)}y_N^J(\sigma)} e^{(N - 1)(A_i + B_i + C_i + D_i + E_i)} \mathbbm{1}_{\{x_N^{J, Z}(\sigma) \in B(x_i, \delta)\}} \Big]e^{(N - 1)\Omega(x_i)}. \label{Equation:Upper_6}
    \end{align}
    Now, we estimate the five terms in \eqref{Equation:Upper_5} individually.
    \begin{enumerate}
        \item[1$^\circ$]  On the event $\{x_N^{J,Z}(\sigma)\in B(x_i,\delta)\}$, since $\log_\eta$ is $(2\eta)^{-1}$-Lipschitz, we have
        \begin{equation}\label{Equation:A_Upper}
            A_i  \le \frac{\delta}{2\eta}.
        \end{equation}
        \item[2$^\circ$] For $C_i$, we see from \eqref{Equation:Upper_Truncation} that
        \begin{equation}\label{Equation:C_Upper}
            C_i \leq \sup_{\sigma, |x| \leq M, \eta > 0}\int_{\mathbb{R}} \big(\log_\eta(\lambda - x) - \log_\eta^K(\lambda - x)\big)\,\mathbb{E}_{J, g}\big[\mu_{N - 1}^{J, g, \sigma}\big](d\lambda) \leq \dfrac{\oldconstant{Constant:Upper_Truncation}M^2}{K^2} = \dfrac{\oldconstant{Constant:Upper_Truncation}}{M^2}.
        \end{equation}
        \item[3$^\circ$] For $D_i$, we see from \eqref{Equation:Universal_Hessian_W1} that
        \begin{equation}\label{Equation:D_Upper}
            D_i \leq \dfrac{\log_\eta(K)}{2\eta} \sup_{\sigma} d_{\text{BL}}\big(\mathbb{E}_{J, g}\big[\mu_{N - 1}^{J, g, \sigma}\big], \mu_{\text{sc}}\big) =:  r_N(\eta, K) =  o_{N}(1)
        \end{equation}
        for each fixed $\eta > 0$ and $K > 0$. 
        \item[4$^\circ$] From \eqref{Equation:Log_Bound}, we see that since $K \geq M + 2$ and $|x_i| \leq M$,
        \begin{equation}\label{Equation:E_Upper}
            E_i \leq \eta.
        \end{equation}
        \item[5$^\circ$] Let us consider the term $B_i = W_{\eta, N}^{J, g}(\sigma, x_i) - \mathbb{E}_{J, g}\big[W_{\eta, N}^{J, g}(\sigma, x_i)\big]$. From \eqref{Equation:Concentration_Hessian_1} and \eqref{Equation:Concentration_Hessian_2} it follows that there exists some $\newconstant\label{Constant:Upper_3} = \oldconstant{Constant:Upper_3}(\kappa_{\text{SG}}, p) > 0$ such that for sufficiently large $N$, 
        \begin{align}
            &\mathbb{E}_{J, g, Z}\Big[e^{-\frac{Nt}{2p(1 - t)}y_N^J(\sigma)} \cdot e^{(N - 1)B_i} \mathbbm{1}_{\{x_N^{J, Z}(\sigma) \in B(x_i, \delta)\}}\Big] \notag \\
            &\leq e^{(N - 1)h} \cdot \mathbb{E}_{J, Z}\Big[e^{-\frac{Nt}{2p(1 - t)}y_N^J(\sigma)} \mathbbm{1}_{\{x_N^{J, Z}(\sigma) \in B(x_i, \delta)\}}\Big] + \mathbb{E}_{J, g,  Z}\Big[e^{(N - 1)B_i} \mathbbm{1}_{\{|B_i| \geq h\}}\Big] \notag \\
            &\leq e^{(N - 1)h} \cdot \mathbb{E}_{J, Z}\Big[e^{-\frac{Nt}{2p(1 - t)}y_N^J(\sigma)} \mathbbm{1}_{\{x_N^{J, Z}(\sigma) \in B(x_i, \delta)\}}\Big] + \Big(\oldconstant{Constant:Concentration}^{-1}e^{\frac{(N - 1)^2\log N}{\oldconstant{Constant:Concentration}\eta^2N^2}} \cdot \oldconstant{Constant:Concentration}^{-1}e^{-\frac{\oldconstant{Constant:Concentration}\eta^2N^2h^2}{\log N}}\Big)^{1/2} \notag \\
            &\leq e^{(N - 1)h} \mathbb{E}_{J, Z}\Big[e^{-\frac{Nt}{2p(1 - t)}y_N^J(\sigma)} \mathbbm{1}_{\{x_N^{J, Z}(\sigma) \in B(x_i, \delta)\}}\Big] + \oldconstant{Constant:Upper_3}^{-1}e^{-\oldconstant{Constant:Upper_3}\eta^2N^2h^2/\log N} \label{Equation:B_Upper}.
        \end{align}
    \end{enumerate}
    Combining \eqref{Equation:Upper_6}, \eqref{Equation:A_Upper}, \eqref{Equation:C_Upper}, \eqref{Equation:D_Upper}, \eqref{Equation:E_Upper}, and \eqref{Equation:B_Upper},  
    and then averaging over $\sigma$, we obtain
    \begin{equation}\label{Equation:Upper_7}
        \mathcal{L}_N(u) \leq e^{-\oldconstant{Constant:Upper_1}NM^2} + e^{(N - 1)\big( \frac{\delta}{2\eta} + \eta + \frac{\oldconstant{Constant:Upper_Truncation}}{M^2} + r_N(\eta, K) \big)} \sum_{i = 1}^n e^{(N - 1)\Omega(x_i)}\big(e^{(N - 1)h}\ell_N(x_i, \delta) + a_N\big).
    \end{equation}
    where $a_N = \oldconstant{Constant:Upper_3}^{-1}e^{-\oldconstant{Constant:Upper_3}\eta^2N^2h^2/\log N}$. From the formula \eqref{Equation:SC_Potential}, we see that there exists an absolute constant $\newconstant\label{Constant:Det_Log_Bound} > 0$ such that $\Omega(x) \leq \oldconstant{Constant:Det_Log_Bound}x^2$ for all $x \in \mathbb{R}$. Therefore,
    \begin{equation}\label{Equation:Upper_9}
        a_N\sum_{i = 1}^n e^{(N - 1)\Omega(x_i)} \leq ne^{(N - 1)\oldconstant{Constant:Det_Log_Bound}M^2 - \oldconstant{Constant:Upper_3}\eta^2N^2h^2/\log N}
    \end{equation}
    Since the cover is finite and independent of $N$, \eqref{Equation:Upper_7}, \eqref{Equation:Upper_9}, and $r_N(\eta, K) \to 0$ imply that
    \begin{equation}\label{Equation:Upper_10}
        \begin{aligned}
            &\limsup_{N \to \infty} \frac{1}{N} \log\mathcal{L}_N(u) \\
            &\leq \max\Big\{-\oldconstant{Constant:Upper_1}M^2, \frac{\delta}{2\eta} + \eta + h + \frac{\oldconstant{Constant:Upper_Truncation}}{M^2} + \sup_{-M\leq x\leq 2u/E_\infty}\Big\{\Omega(x) + \limsup_{N\to\infty}\frac{1}{N}\log\ell_N(x,\delta)\Big\}\Big\}.
        \end{aligned}
    \end{equation}
    Here the term in \eqref{Equation:Upper_9} is superexponentially small for fixed $M, \eta, h > 0$. Letting $M \uparrow \infty$ in \eqref{Equation:Upper_10} gives
    \[\limsup_{N\to\infty}\frac{1}{N}\log\mathcal{L}_N(u) \leq \frac{\delta}{2\eta} + \eta + h + \sup_{x\leq 2u/E_\infty}\Big\{\Omega(x) + \limsup_{N\to\infty}\frac{1}{N}\log\ell_N(x,\delta)\Big\}.\]
    For fixed $\eta, h > 0$, let $\delta \downarrow 0$, and then let $\eta, h \downarrow 0$. The limit in $\delta$ exists in $[-\infty, \infty)$ because $\ell_N(x,\delta)$ is nonincreasing in $\delta$. This proves the proposition.
\end{proof}

\subsection{Lower Bound}

\begin{proposition}\label{Theorem:Lower}
    If $\{J_I\}_{I \in [N]^p, N \in \mathbb{N}}$ satisfies \eqref{Equation:SG} with uniform $\kappa_{\text{SG}}$, then, for $u \in \mathbb{R}$, 
    \[\liminf_{N \to \infty}\dfrac{1}{N}\log \mathcal{L}_N(u) \geq \sup_{x < 2u/E_\infty} \Big\{\Omega(x) + \lim_{\delta \to 0}\liminf_{N \to \infty} \dfrac{1}{N}\log \ell_N(x, \delta)\Big\}.\]
\end{proposition}

\begin{proof}
    Fix $h > 0$ and let $M \geq 1 + |2u/E_\infty|$, $K \geq M + 2$ be parameters to be chosen sufficiently large. Fix $-M \leq x < 2u/ E_\infty$. Choose $\epsilon > 0$ such that $x + \epsilon < 2u/E_\infty$ and let $0 < \delta < \epsilon$.  Then for sufficiently large $N$, $B(x, \delta) \subseteq U_N(u).$  Let $\eta = \eta(h, t)>0$  be chosen later.

    We introduce the following notations
    \begin{align*}
        A &= \int_{\mathbb{R}} \big(\log |\lambda - x_N^{J, Z}(\sigma)| - \log_\eta(\lambda -  x_N^{J, Z}(\sigma))\big)\,\mu_{N - 1}^{J, g, \sigma}(d\lambda), \\
        B &= \int_{\mathbb{R}}\big(\log_\eta(\lambda - x_N^{J, Z}(\sigma)) - \log_\eta(\lambda - x)\big)\,\mu_{N - 1}^{J, g, \sigma}(d\lambda), \\
        C &= \int_{\mathbb{R}}\log_\eta(\lambda - x)\big(\mu_{N - 1}^{J, g, \sigma} - \mathbb{E}_g\big[\mu_{N - 1}^{J, g, \sigma}\big]\big)(d\lambda), \\
        D &= \int_{\mathbb{R}}\log_\eta(\lambda - x)\big(\mathbb{E}_g\big[\mu_{N - 1}^{J, g, \sigma}\big] - \mathbb{E}_{J, g}\big[\mu_{N - 1}^{J, g, \sigma}\big]\big)(d\lambda), \\
        E &= \int_{\mathbb{R}} \big(\log_\eta(\lambda - x) - \log_\eta^K(\lambda - x)\big)\,\mathbb{E}_{J, g}\big[\mu_{N - 1}^{J, g, \sigma}\big](d\lambda), \\
        F &= \int_{\mathbb{R}} \log_\eta^K(\lambda - x)\,\big(\mathbb{E}_{J, g}\big[\mu_{N - 1}^{J, g, \sigma}\big] - \mu_{\text{sc}}\big)(d\lambda).
    \end{align*}

    Then
    \begin{equation}\label{Equation:Lower_Decomposition}
        \int_{\mathbb R}\log|\lambda- x_N^{J,Z}(\sigma)| \mu_{N-1}^{J,g,\sigma}(d \lambda) = A + B + C + D + E + \int_{\mathbb R}\log_\eta^K(\lambda - x) \mu_{\mathrm{sc}}(d\lambda).
    \end{equation}
    Conditioning on $(\sigma, J, Z)$, the terms $D$ and $E$ do not depend on $g$, while $\mathbb E_g[C] = 0$. Moreover, $E \geq 0$ since $\log_\eta(\lambda - x) - \log_\eta^K(\lambda - x) \geq 0$. Hence, by Jensen's inequality, 
    \[\mathbb{E}_g\big[e^{(N-1)(A + B + C + D + E + F)}\big] \geq \exp\big((N - 1) \big(\mathbb E_g[A+B]+D+ F \big) \big).\]
    Moreover, since $K \geq M + 2$ and $|x| \leq M$, we have by \eqref{Equation:Log_Bound} that
    \[\int_{\mathbb R}\log_\eta^K(\lambda-x) \mu_{\mathrm{sc}}(d\lambda) \geq \int_{\mathbb R}\log|\lambda-x| \mu_{\mathrm{sc}}(d\lambda) = \Omega(x).\]
    Therefore, for every fixed $\sigma$ and for all sufficiently large $N$, 
    \begin{align}
        &\mathbb{E}_{J, g, Z}\Big[e^{-\frac{Nt}{2p(1 - t)}y_N^J(\sigma)}\big|\det \big(M_{N - 1}^{J, g}(\sigma) - x_N^{J, Z}(\sigma)I_{N - 1}\big)\big|\cdot \mathbbm{1}_{\{x_N^{J, Z}(\sigma) \in U_N(u)\}}\Big] \notag \\
        &\geq e^{(N - 1)\Omega(x)} \mathbb{E}_{J,  Z}\Big[e^{-\frac{Nt}{2p(1 - t)}y_N^J(\sigma)} \cdot e^{(N - 1)(\mathbb{E}_g[A + B] + D + F)} \cdot \mathbbm{1}_{\{x_N^{J, Z}(\sigma) \in B(x, \delta)\}}\Big] , \label{Equation:Lower_1}
    \end{align}
    Now, we estimate the four terms in \eqref{Equation:Lower_1} individually.
    \begin{enumerate}
        \item[1$^\circ$] For $A$, we see from Lemma \ref{Lemma:Truncation_of_Log} that uniform over deterministic realizations $J \in \mathbb{R}^{[N]^p}$, $\sigma \in \sqrt{N}\mathbb{S}^{N - 1}$ and $x \in \mathbb{R}$, there exists $\eta = \eta(h, t) > 0$ so small so that
        \begin{equation}\label{Equation:A_Lower_Pre}
            \sup_{N \geq 2} \sup_{J, \sigma, x}  
            \int_{\mathbb{R}} \bigl( \log_\eta(\lambda-x)-\log|\lambda-x| \bigr) \mathbb{E}_g\bigl[\mu_{N-1}^{J,g,\sigma}\bigr](d\lambda) \leq h.
        \end{equation} 
        Taking $x = x_N^{J,Z}(\sigma)$ in \eqref{Equation:A_Lower_Pre} then gives
        \begin{equation} \label{Equation:A_Lower}
            \mathbb E_g[A] = -\int_{\mathbb R} \bigl( \log_\eta(\lambda-x_N^{J,Z}(\sigma)) - \log|\lambda-x_N^{J,Z}(\sigma)|\bigr) \mathbb E_g\bigl[\mu_{N-1}^{J,g,\sigma}\bigr](d\lambda) \geq -h.
        \end{equation}

        \item[2$^\circ$] On the event $\{x_N^{J, Z}(\sigma) \in B(x, \delta)\}$, since $\log_\eta$ is $(2\eta)^{-1}$-Lipschitz, we have
        \begin{equation}\label{Equation:B_Lower}
             B \geq -\dfrac{\delta}{2\eta}.
        \end{equation}
        
        \item[3$^\circ$] For $F$, we see from \eqref{Equation:Universal_Hessian_W1} that
        \begin{equation}\label{Equation:F_Lower}
            F \geq -\dfrac{\log_\eta(K)}{2\eta} \sup_{\sigma} d_{\text{BL}}\big(\mathbb{E}_{J, g}\big[\mu_{N - 1}^{J, g, \sigma}\big], \mu_{\text{sc}}\big) = -r_N(\eta, K) = o_{N}(1)
        \end{equation}
        for each fixed $\eta > 0$ and $K > 0$.
    
        \item[4$^\circ$] Finally, we analyze the term $D$. Applying the same argument as in Lemma \ref{Lemma:Lip_Concentration} to the mapping $J \mapsto \mathbb{E}_g\big[W_{\eta, N}^{J, g}(\sigma, x)\big]$, we obtain the same concentration estimate as in \eqref{Equation:Concentration_Hessian_1}. More precisely, for all $h \geq 0$,
        \[\sup_{\sigma, x} \mathbb{P}_J(|D| \geq h) \leq \oldconstant{Constant:Concentration}^{-1}\exp\Big(-\frac{\oldconstant{Constant:Concentration}\eta^2N^2h^2}{\log N}\Big).\]
        Therefore,
        \begin{align}
            &\mathbb{E}_{J, Z}\Big[e^{-\frac{Nt}{2p(1 - t)}y_N^J(\sigma)} \cdot e^{(N - 1)D} \cdot \mathbbm{1}_{\{x_N^{J, Z}(\sigma) \in B(x, \delta)\}}\Big] \notag \\
            &\geq e^{-(N - 1)h} \cdot \mathbb{E}_{J, Z}\Big[e^{-\frac{Nt}{2p(1 - t)}y_N^J(\sigma)} \cdot \mathbbm{1}_{\{x_N^{J, Z}(\sigma) \in B(x, \delta)\}} \cdot \mathbbm{1}_{\{|D| \leq h\}}\Big] \notag \\
            &\geq e^{-(N - 1)h}\Big(\mathbb{E}_{J, Z}\Big[e^{-\frac{Nt}{2p(1 - t)}y_N^J(\sigma)} \cdot \mathbbm{1}_{\{x_N^{J, Z}(\sigma) \in B(x, \delta)\}}\Big] - \oldconstant{Constant:Concentration}^{-1}e^{-\oldconstant{Constant:Concentration}\eta^2N^2h^2/\log N}\Big). \label{Equation:D_Lower}
        \end{align}  
    \end{enumerate}
    Combining \eqref{Equation:Lower_1}, \eqref{Equation:A_Lower}, \eqref{Equation:B_Lower}, \eqref{Equation:F_Lower}, and \eqref{Equation:D_Lower}, and then averaging over $\sigma$, gives
    \begin{equation}\label{Equation:Lower_Main}
        \mathcal{L}_N(u) \geq \exp\Big[(N - 1)\Big(\Omega(x)-2h-\frac{\delta}{2\eta} - r_N(\eta, K) \Big)\Big]\big(\ell_N(x,\delta) - b_N\bigr),
    \end{equation}
    where $b_N = \oldconstant{Constant:Concentration}^{-1}e^{-\oldconstant{Constant:Concentration}\eta^2N^2h^2/\log N}$. Set
    \[\ell(x,\delta) := \liminf_{N\to\infty}\frac{1}{N}\log \ell_N(x,\delta).\]
    We claim that
    \begin{equation}\label{Equation:Lower_Fixed_Parameters}
        \liminf_{N\to\infty}\frac{1}{N}\log\mathcal{L}_N(u) \geq \Omega(x)-2h-\frac{\delta}{2\eta} + \ell(x,\delta).  
    \end{equation}
    If $\ell(x,\delta) = -\infty$, the desired lower bound for this fixed $x$ and $\delta$ is automatic. Suppose instead that $\ell(x,\delta) > -\infty$. Then, for every $\rho>0$ and all sufficiently large $N$,
    \[\ell_N(x,\delta)\ge \exp\bigl(N(\ell(x,\delta) - \rho)\bigr).\]
    Since $b_N$ is superexponentially small, $\ell_N(x,\delta) - b_N > 0$ for all sufficiently large $N$, and
    \[\liminf_{N\to\infty}\frac{1}{N} \log\bigl(\ell_N(x,\delta) - b_N\bigr) = \liminf_{N\to\infty}\frac{1}{N}\log \ell_N(x,\delta) = \ell(x,\delta).\]
    Hence, Claim \eqref{Equation:Lower_Fixed_Parameters} now follows from \eqref{Equation:Lower_Main}. For the fixed $h>0$ and the corresponding fixed $\eta = \eta(h,t)>0$, let $\delta\downarrow0$ in \eqref{Equation:Lower_Fixed_Parameters}. This yields
    \[\liminf_{N\to\infty}\frac{1}{N}\log\mathcal L_N(u) \geq \Omega(x) - 2h + \lim_{\delta\downarrow 0} \ell(x, \delta)\]
    Next let $h\downarrow0$. Since $x<2u/E_\infty$ was arbitrary, taking the supremum over such $x$ proves the proposition.
\end{proof}

\section{Convergence of Joint Moment Generating Function}\label{Subsection:MGF}

\noindent Set the random variable
\begin{equation}\label{Equation:Def_X}
    X_N^J(\sigma) = \frac{H_N^J(\sigma)}{\sqrt{N(N-1)}}
\end{equation}
and recall $y_N^J(\sigma)$ from \eqref{Equation:Def_y}. Consider the joint moment generating function
\begin{align}
    \Lambda_N(\alpha, \beta) &:= \dfrac{1}{N}\log\mathbb{E}_{\sigma, J} \big[\exp\big(N\big\langle (\alpha, \beta), (X_N^J(\sigma), y_N^J(\sigma))\big\rangle\big)\big] \notag \\
    &= \dfrac{1}{N}\log\mathbb{E}_{\sigma, J}\Big[\exp\Big(\alpha \sqrt{\frac{N}{N - 1}}H_N^J(\sigma) + \beta \|B(\sigma)^\top\nabla H_N^J(\sigma)\|_2^2\Big)\Big]. \label{Equation:MGF_H_Grad}
\end{align}
The main result of this section is the following convergence:

\begin{theorem}
    If $p \geq 3$ and $\{J_I\}_{I \in [N]^p, N \in \mathbb{N}}$ satisfies \eqref{Equation:SSG}, then 
    \begin{equation}\label{Equation:Convergence_of_MGF}
        \Lambda(\alpha, \beta) := \lim_{N \to \infty} \Lambda_N(\alpha, \beta) = \begin{cases}
            \frac{\alpha^2}{2} - \frac{1}{2}\log(1 - 2\beta p), & \beta < \frac{1}{2p}, \\
            \infty, & \beta \geq \frac{1}{2p}.
        \end{cases}
    \end{equation}
\end{theorem}

\subsection{Auxiliary Bounds}
\noindent
For a random variable $X$ satisfying \eqref{Equation:SG}, we define the complex moment generating function by
\[M_X(z) = \mathbb{E}\big[e^{zX}\big], \quad \forall z \in \mathbb{C}.\]
Note that subgaussianity readily ensures that $M_X$ is entire by the dominated convergence theorem.

\begin{lemma}
    If $X$ satisfies \eqref{Equation:SSG}, then there exists absolute constant $\newconstant\label{Constant:MGF_Bound} > 0$ such that for all $x, y \in \mathbb{R}$,
    \begin{equation}\label{Equation:SSG_MGF}
        |M_X(x + iy)| \leq \exp\Big(\dfrac{x^2 - y^2}{2} + \oldconstant{Constant:MGF_Bound}x^4 + h(y)\Big),
    \end{equation}
    where $h(y) = \min\{y^2/2, \oldconstant{Constant:MGF_Bound}y^4\}$.
\end{lemma}

\begin{proof}
    Since $M_X(0) = 1 \neq 0$, we can choose an analytic branch $K_X(z) = \log M_X(z)$ for $|z|$ sufficiently small. By condition \eqref{Equation:SSG}, $\mathbb{E}[X] = \mathbb{E}[X^3] = 0$, and $\mathbb{E}[X^2] = 1$, there exists $\delta > 0$ and $C > 0$ so that
    \[\Big|K_X(z) - \dfrac{z^2}{2}\Big| \leq C|z|^4 \quad \text{for all} \quad |z| < \delta.\]
    We can enlarge $C > 0$ and shrink $\delta > 0$ if necessary so that
    \begin{equation}\label{Equation:SSG_1}
        |M_X(x + iy)| \leq \exp\Big(\dfrac{x^2 - y^2}{2} + C(x^4 + y^4)\Big) \quad \text{for all} \quad |x|, |y| < \delta.
    \end{equation}
   On the other hand, \eqref{Equation:SSG} also implies that for all $x, y \in \mathbb{R}$,
    \begin{equation}\label{Equation:SSG_2}
        |M_X(x + iy)| \leq M_X(x) \leq \exp\Big(\dfrac{x^2}{2}\Big) \leq \exp\Big(\dfrac{x^2 - y^2}{2} + \dfrac{y^2}{2}\Big).
    \end{equation}
    If we take $\oldconstant{Constant:MGF_Bound} = \max\{C, 1/(2\delta^2)\}$, then 
    \begin{itemize}
        \item for $|y| < \delta$ and $|x| < \delta$, \eqref{Equation:SSG_1} and \eqref{Equation:SSG_2} together gives \eqref{Equation:SSG_MGF};
        \item for $|y| < \delta$ and $|x| \geq \delta$, from \eqref{Equation:SSG_2},
        \[|M_X(x + iy)| \leq M_X(x) \leq \exp\Big(\dfrac{x^2 - y^2}{2} + \dfrac{\delta^2}{2}\Big) \leq \exp\Big(\dfrac{x^2 - y^2}{2} + \oldconstant{Constant:MGF_Bound}x^4\Big);\]
        \item for $|y| \geq \delta$, we have $\oldconstant{Constant:MGF_Bound}y^4 \geq y^2/2$ and \eqref{Equation:SSG_MGF} follows from \eqref{Equation:SSG_2}.
    \end{itemize}
    This concludes the proof of the lemma.
\end{proof}

\begin{proposition}\label{Proposition:General_Bound}
    For each $N \in \mathbb{N}$, let $X_1, \dots, X_{m_N}$ be independent random variables satisfying \eqref{Equation:SSG}. Let $\{a_j\}_{j = 1}^{m_N} \subseteq \mathbb{R}$ and $\{b_j\}_{j = 1}^{m_N} \subseteq \mathbb{R}^{d_N}$ be deterministic coefficients satisfying
    \begin{equation}\label{Equation:General_Covariance}
        \sum_{j = 1}^{m_N} a_j^2 = A_N, \quad \sum_{j = 1}^{m_N} a_jb_j = 0 \quad \text{and} \quad \sum_{j=1}^{m_N}b_jb_j^\top = \tau_N^2I_{d_N}.
    \end{equation}
    Moreover, there exists $\newconstant\label{Constant:General_Bound} > 0$ such that
    $\tau_N^2 \leq \oldconstant{Constant:General_Bound}$. Define for $\alpha \in \mathbb{R}$, $\gamma > 0$ and $v \in \mathbb{R}^{d_N}$,
    \[R_N(\alpha, \gamma, v) = \mathbb{E}\Big[\exp\Big(\alpha\sum_{j = 1}^{m_N} a_jX_j - \gamma \Big\|v + \sum_{j = 1}^{m_N}b_jX_j\big\|_2^2\Big)\Big].\]
    Last, set
    \[\rho_N = \max_{1 \leq j \leq m_N} \big(|a_j| + \|b_j\|_2\big) \quad \text{and} \quad \Delta_N = \sum_{j = 1}^{m_N} \big(a_j^4 + \|b_j\|_2^4\big).\]
    \begin{enumerate}[label = (\roman*)]
        \item There exists $\newconstant\label{Constant:UB_MGF} = \oldconstant{Constant:UB_MGF}(\alpha, \gamma, \oldconstant{Constant:General_Bound}) > 0$ such that
        \begin{equation}\label{Equation:UB_General}
            \sup_{v \in \mathbb{R}^{d_N}}\log R_N(\alpha, \gamma, v) \leq \frac{\alpha^2A_N}{2} - \frac{d_N}{2}\log(1 + 2\gamma\tau_N^2) + \oldconstant{Constant:UB_MGF}\Delta_N
        \end{equation}
        \item If $\rho_N = o_N(1)$ and $\Delta_N = O_N(1)$, then
        \begin{equation}\label{Equation:LB_General}
            \log R_N(\alpha, \gamma, 0) = \frac{\alpha^2A_N}{2} - \frac{d_N}{2}\log(1 + 2\gamma\tau_N^2) + o_N(N).
        \end{equation}
    \end{enumerate}
\end{proposition}

\begin{proof}
    Consider the random vector $g \sim \texttt{Normal}(0, I_{d_N})$,
    \[x_j = \alpha a_j, \quad y_j = \langle b_j, g\rangle, \quad \text{and} \quad \xi_N = \sqrt{\dfrac{2\gamma}{1 + 2\gamma \tau_N^2}}.\]
    Denote the function
    \begin{equation}\label{Equation:General_Bound_1}
        \mathcal{F}_N(g) = \exp\Big(-\dfrac{1}{2}\big(\alpha^2A_N + \tau_N^2\|g\|_2^2\big)\Big)\prod_{j = 1}^{m_N}M_{X_j}(x_j + iy_j)
    \end{equation}
    By linearizing the quadratic term in $R_N(\alpha, \beta, v)$ with the identity $e^{-\|v\|_2^2/2} = \mathbb{E}_g[e^{i\langle v, g\rangle}]$, we have
    \begin{align}
        \log R_N(\alpha, \beta, v) &= \log\mathbb{E}_{g}\Big[e^{i\sqrt{2\gamma}\langle v, g\rangle} \prod_{j = 1}^{m_N}M_{X_j}(x_j + i\sqrt{2\gamma}y_j)\Big] \notag \\
        &= \frac{\alpha^2A_N}{2} - \frac{d_N}{2}\log(1 + 2\gamma\tau_N^2) + \log \mathbb{E}_g\big[e^{i\sqrt{2\gamma}\langle v, g\rangle} \mathcal{F}_N(\xi_Ng)\big], \label{Equation:General_Bound_2}
    \end{align}
    where \eqref{Equation:General_Bound_2} holds by the identities \eqref{Equation:General_Covariance}.
    
    First, we claim that there exists $q > 1$ and $C > 0$ such that
    \begin{equation}\label{Equation:General_Bound_3}
        \mathbb{E}\big[|\mathcal{F}_N(\xi_Ng)|^q\big] \leq e^{C\Delta_N}
    \end{equation}
    We see from the inequality \eqref{Equation:SSG_MGF} induced by strict subgaussianity that
    \begin{align}
        &\log \mathbb{E}_g\big[|\mathcal{F}_N(\xi_Ng)|^q\big] \notag \\
        &\leq -\dfrac{q}{2}\big(\alpha^2A_N + \tau_N^2\|\xi_Ng\|_2^2\big) +  \mathbb{E}_g\Big[\exp\Big(q\sum_{j = 1}^{m_N} \dfrac{x_j^2 - (\rho _Ny_j)^2}{2} + \oldconstant{Constant:MGF_Bound}x_j^4 + h(\xi_Ny_j)\Big)\Big] \notag \\
        &= \oldconstant{Constant:MGF_Bound}q\sum_{j = 1}^{m_N}x_j^4 + \log \mathbb{E}_g\Big[\exp\Big(q\sum_{j = 1}^{m_N} h(\xi_Ny_j)\Big)\Big] \label{Equation:General_Bound_4},
    \end{align}
    where \eqref{Equation:General_Bound_4} holds by the identities \eqref{Equation:General_Covariance}. 
    Next, we choose $q > 1$ small enough so that
    \begin{equation}\label{Equation:General_Bound_5}
        q\Big(\dfrac{\gamma \oldconstant{Constant:General_Bound}}{1 + 2\gamma \oldconstant{Constant:General_Bound}}\Big) < \frac{1}{2}.
    \end{equation}
    For the last term in \eqref{Equation:General_Bound_4}, we shall apply the Gaussian Brascamp-Lieb inequality \eqref{Equation:GBL} with
    \[m = m_N, \quad d = d_N, \quad u_j = \dfrac{b_j}{\|b_j\|_2}, \quad \lambda_j = \dfrac{\|b_j\|_2^2}{\tau_N^2}, \quad f_j(x) = \exp\Big(\dfrac{qh(\xi_N\|b_j\|_2x)}{\lambda_j}\Big),\]
    where we may assume without lose of generality that $b_j \neq 0$. We see then
    \begin{equation}\label{Equation:General_Bound_6}
        \log \mathbb{E}_g\Big[\exp\Big(\sum_{j = 1}^{m_N} h(\xi_Ny_j)\Big)\Big] = \log \mathbb{E}_Z\Big[\prod_{j = 1}^{m_N} f_j(\langle u_j, g\rangle)^{\lambda_j}\Big] \leq \sum_{j = 1}^{m_N} \lambda_j \log \mathbb{E}_Z[f_j(Z)]
    \end{equation}
    Note that by the elementary inequalities $\log(1 + x) \leq x$, $e^x - 1\leq xe^x$, and the definition of the function $h(y) = \min\{y^2/2, \oldconstant{Constant:MGF_Bound}y^4\}$, we see for our choice of $q > 1$ in \eqref{Equation:General_Bound_5},
    \begin{align}
        \lambda_j \log \mathbb{E}_Z[f_j(Z)] \leq \lambda_j (\mathbb{E}_Z[f_j(Z)] - 1) &\leq \mathbb{E}_Z\Big[qh(\xi_N\|b_j\|_2Z)\exp\Big(\dfrac{qh(\xi_N\|b_j\|_2Z)}{\lambda_j}\Big)\Big] \notag \\
        &\leq \mathbb{E}_Z\Big[\oldconstant{Constant:MGF_Bound}q(\xi_N\|b_j\|_2Z)^4\exp\Big(\dfrac{q\xi_N^2\|b_j\|_2^2Z^2}{2\lambda_j}\Big)\Big] \notag \\
        &= \oldconstant{Constant:MGF_Bound}q\xi_N^4\mathbb{E}_Z\Big[Z^4\exp\Big(q\Big(\dfrac{\gamma \tau_N^2}{1 + 2\gamma \tau_N^2}\Big)Z^2\Big)\Big]\|b_j\|_2^4 \notag \\
        &\leq \oldconstant{Constant:MGF_Bound}q\xi_N^4\mathbb{E}_Z\Big[Z^4\exp\Big(q\Big(\dfrac{\gamma \oldconstant{Constant:General_Bound}}{1 + 2\gamma \oldconstant{Constant:General_Bound}}\Big)Z^2\Big)\Big]\|b_j\|_2^4 \label{Equation:General_Bound_7}
    \end{align}
    The claim \eqref{Equation:General_Bound_3} follows directly by combining \eqref{Equation:General_Bound_4}, \eqref{Equation:General_Bound_6}, and \eqref{Equation:General_Bound_7}.
    
    To show the upper bound, we see that \eqref{Equation:UB_General} follows from combining \eqref{Equation:General_Bound_2}, the claim \eqref{Equation:General_Bound_3} with $q = 1$ and
    \[\big|\mathbb{E}_g\big[e^{i\sqrt{2\gamma}\langle v, g\rangle} \mathcal{F}_N(\xi_Ng)\big]\big| \leq \mathbb{E}_g\big[|\mathcal{F}_N(\xi_Ng)|\big].\]
    
    For the second item, let $\kappa_j = (\log M_{X_j})^{(4)}(0)$ be the fourth cumulant of $X_j$ (which, say, are uniformly bounded by $\kappa$). By \eqref{Equation:SSG}, there exists $\delta, C > 0$  such that for all $|z| < \delta$
    \[\log M_{X_j}(z) = \dfrac{z^2}{2} + \dfrac{\kappa_j}{24}z^4 + Q_j(z) \quad \text{and} \quad \sup_{N \in \mathbb{N}}\sup_{1 \leq j \leq m_N}|Q_j(z)| \leq C|z|^5.\]
    Define
    \[\mathcal{D}_N(g) = \dfrac{1}{24}\sum_{j = 1}^{m_N} \kappa_j(x_j + i\xi_Ny_j)^4 \quad \text{and} \quad \mathcal{E}_N = \Big\{\max_{1 \leq j \leq m_N}|x_j + i\xi_N y_j| < \delta\Big\}.\]
    We claim that on the event $\mathcal{E}_N$,
    \begin{equation}\label{Equation:LB_General_1}
        \mathcal{F}_N(\xi_Ng) \to \exp\big(\mathbb{E}_g\big[\mathcal{D}_N(g)\big]\big) \quad \text{in $\mathcal{L}^1$}. 
    \end{equation}
    Note that on $\mathcal{E}_N$, one has from the definition \eqref{Equation:General_Bound_1}, the identities \eqref{Equation:General_Covariance}, that
    \begin{equation}\label{Equation:LB_General_2}
        \mathcal{F}_N(\xi_Ng) = \exp\Big[\mathcal{D}_N(g) + O\Big(\sum_{j = 1}^{m_N}|x_j + i\xi_Ny_j|^5\Big)\Big]
    \end{equation}
    By Gaussian moment bounds, there exists $\newconstant\label{Constant:LB_MGF} = \oldconstant{Constant:LB_MGF}(\alpha, \gamma, \oldconstant{Constant:General_Bound}) > 0$ such that for $k = 4, 5, 6$,
    \begin{equation}\label{Equation:LB_General_3}
        \mathbb{E}_g\Big[\sum_{j = 1}^{m_N} |x_j + i\xi_Ny_j|^k\Big] \leq 2^k\sum_{j = 1}^{m_N} \big(|\alpha|^k|a_j|^k + \xi_N^k\mathbb{E}_g\big[|\langle b_j, g\rangle|^k\big]\big) \leq \oldconstant{Constant:LB_MGF}\rho_N^{k - 4}\Delta_N
    \end{equation}
    In particular, $\mathbb{P}_g(\mathcal{E}_N) \to 1$ by Markov's inequality and the remainder term in \eqref{Equation:LB_General_2} satisfies $\sum_{j = 1}^{m_N}|x_j + i\xi_Ny_j|^5 \to 0$ in probability. Moreover, by Gaussian Poincar\'{e} inequality (see \cite[Theorem 3.20]{boucheron2013concentration}), applied to the real and imaginary parts,
    \begin{align}
        \mathbb{E}_g\big[\big|\mathcal{D}_N(g) - \mathbb{E}_g[\mathcal{D}_N(g)]\big|^2\big] &\leq \xi_N^2 \mathbb{E}_g\big[\|\nabla\mathcal{D}_N(g)\|_2^2\big] \notag \\
        &= \xi_N^2 \mathbb{E}_g\Big[\Big\|\dfrac{1}{6}\sum_{j = 1}^{m_N}\kappa_j(x_j + i\xi_Ny_j)^3\xi_Nb_j\Big\|_2^2\Big] \notag \\
        &\leq \dfrac{\xi_N^4\kappa^2\tau_N^2}{6}\mathbb{E}_g\Big[\sum_{j = 1}^{m_N} |x_j + i\xi_Ny_j|^6\Big] \notag \\
        &\leq \dfrac{\xi_N^4\kappa^2\tau_N^2}{6} \cdot \oldconstant{Constant:LB_MGF}\rho_N^2\Delta_N = o_N(1), \label{Equation:LB_General_4}
    \end{align}
    where the last inequality holds by \eqref{Equation:LB_General_3} with $k = 6$. Combining \eqref{Equation:LB_General_2} and \eqref{Equation:LB_General_4}, we see that $\mathcal{F}_N(\xi_Ng) \to \exp\big(\mathbb{E}_g\big[\mathcal{D}_N(g)\big]\big)$ in probability. Since $\Delta_N = O_N(1)$, \eqref{Equation:General_Bound_3} further implies that the sequence $\mathcal{F}_N(\xi_g)$ is $\mathcal{L}^q$-bounded, thus concludes \eqref{Equation:LB_General_1}. Moreover, the $\mathcal{L}^q$-boundedness of $\mathcal{F}_N(\xi_Ng)$ and boundedness of $\mathcal{D}_N(g)$ further implies
    \[\mathbb{E}_g[\mathcal{F}_N(\xi_Ng)] = \exp\big(\mathbb{E}_g\big[\mathcal{D}_N(g)\big]\big)(1 + o_N(1)).\]
    Plugging this in \eqref{Equation:General_Bound_2} and setting $v = 0$ will then give us the desired asymptote \eqref{Equation:LB_General}.
\end{proof}

\subsection{Upper Bound}

\noindent Before applying Proposition \ref{Proposition:General_Bound} to establish the upper bound in \eqref{Equation:Convergence_of_MGF}, we need to introduce some notation. First, we set
\[a_I(\sigma) = \sqrt{\dfrac{N}{N - 1}}\dfrac{\sigma_I}{N^{(p - 1)/2}} \in \mathbb{R} \quad \text{and} \quad b_I(\sigma) = \dfrac{B(\sigma)^\top\nabla\sigma_I}{N^{(p - 1)/2}} \in \mathbb{R}^{N - 1}.\]
For a configuration $\sigma \in \sqrt{N}\mathbb{S}^{N - 1}$, we define the set of delocalized and localized coordinate of $\sigma$ by
\[L(\sigma) = \big\{i \in [N] \,\big\vert\, |\sigma_i| \geq \log N\big\} \subseteq [N] \quad \text{and} \quad D(\sigma) = [N] \setminus L(\sigma),\]
where $|L(\sigma)| \leq r_N := \lfloor N/(\log N)^2\rfloor = o_N(N)$. Additionally, consider the subspace
\[E_D(\sigma) = \big\{v \in \mathbb{R}^N \colon \langle v, \sigma \rangle = 0\,\,\mbox{and}\,\, v_i = 0 \text{ for all } i \in L(\sigma)\big\} \subseteq \{\sigma\}^\perp\]
of dimension $N - |L(\sigma)| - 1$. Let $B_D(\sigma)$ be a $N \times (N - |L(\sigma)| - 1)$ semi-orthogonal matrix consisting of an orthonormal basis of $E_D(\sigma)$. Then, 
\begin{equation}\label{Equation:B_D_and_B}
    \|B(\sigma)^\top v\|_2 \geq \|B_D(\sigma)^\top v\|_2, \quad \forall v \in \mathbb{R}^N
\end{equation}

\begin{remark}\label{add:remark1}
    \rm The matrix $B_D(\sigma)$ has two useful properties that will be essential in our proofs later. First, for any $v \in \mathbb{R}^N$ vanishing on $D(\sigma)$, we have $B_D(\sigma)^\top v = 0$. Second, if $v \in \mathbb{R}^N$ is supported on $D(\sigma)$ with $v_i = \sigma_i$ for $i \in D(\sigma)$, then $B_D(\sigma)^\top v = 0.$ The former holds since any element in $E_D(\sigma)$ is supported on $D(\sigma)$ and this readily implies that $v \in E_D(\sigma)^\perp$ so that $B_D(\sigma)^\top v = 0.$ The latter holds true due to $B_D(\sigma)^\top \sigma = 0$.
\end{remark}

Set $[N]_\text{hi}^p$ to be the collection of tensor index $I$ that has at most $1$ entries in $D(\sigma)$ and $[N]_{\text{lo}}^p = [N]^p \setminus [N]_\text{hi}^p$. We split the Hamiltonian $H_N^J(\sigma) = H_N^{\text{hi}}(\sigma) + H_N^{\text{lo}}(\sigma)$ into high- and low-influence parts, where each part only contains tenser indices that are in the corresponding subset.  Set
\[c_I(\sigma) = \dfrac{B_D(\sigma)^\top\nabla \sigma_I}{N^{(p - 1)/2}} \in \mathbb{R}^{N - |L(\sigma)| - 1} \quad \text{and} \quad \eta = \eta(\sigma) = \dfrac{1}{N}\sum_{i \in L(\sigma)} \sigma_i^2 \in [0, 1],\]

\begin{lemma}
    For all $\sigma \in \sqrt{N}\mathbb{S}^{N - 1}$, the following identities hold:
    \begin{equation}\label{Equation:Covariance_Ord_H}
        \sum_{I \in [N]^p} a_I(\sigma)^2 = \dfrac{N^2}{N - 1}, \,\,\sum_{I \in [N]^p} a_I(\sigma)b_I(\sigma) = 0, \,\,\text{and} \,\, \sum_{I \in [N]^p} b_I(\sigma)b_I(\sigma)^\top = pI_{N - 1},
    \end{equation}
    and also
    \begin{equation}\label{Equation:Ub_MGF_1}
        \begin{aligned}
            \sum_{I \in [N]_{\text{lo}}^p} a_I(\sigma)^2 = \dfrac{N^2s(\eta)}{N - 1}, \,\,\sum_{I \in [N]_{\text{lo}}^p} a_I(\sigma)c_I(\sigma) = 0, \,\,\sum_{I \in [N]_{\text{lo}}^p} c_I(\sigma)c_I(\sigma)^\top = p\big(1 - \eta^{p - 1}\big)I_{N - |L(\sigma)| - 1},
        \end{aligned}
    \end{equation}
    where $s(\eta) = 1 - \eta^p - p\eta^{p - 1}(1 - \eta)$.
\end{lemma}

\begin{proof}
    Both \eqref{Equation:Covariance_Ord_H} and \eqref{Equation:Ub_MGF_1} can be proved by the same computation as in Proposition \ref{Proposition:Conditional_Covariance}. We only provide the details for \eqref{Equation:Ub_MGF_1}. Fix $\sigma\in \sqrt{N}\mathbb{S}^{N-1}$. Set, for any $x,y\in \mathbb{R}^N,$
\begin{align*}
    C(x,y)&=\mathbb{E}\big[H_N^{\text{lo}}(x)H_N^{\text{lo}}(y)] = \sum_{I \in [N]_{\text{lo}}^p} \dfrac{x_Iy_I}{N^{p - 1}}\\
    &= N\big(R(x, y)^p - R_L(x, y)^p - pR_L(x, y)^{p - 1}R_D(x, y)\big),
\end{align*}
    where 
    \[R(x, y):= \dfrac{1}{N}\sum_{i = 1}^N x_iy_i, \,\, R_L(x, y):= \dfrac{1}{N}\sum_{i \in L(\sigma)} x_iy_i, \,\, \text{and} \,\, R_D(x, y):= R(x, y) - R_L(x, y).\]
    By definition, \eqref{Equation:Ub_MGF_1} is obtained by
    \[\sum_{I \in [N]_{\text{lo}}^p} a_I(\sigma)^2 = \dfrac{N}{N - 1}C(\sigma, \sigma) = \dfrac{N^2}{N - 1}\big(1 - \eta^p - p\eta^{p - 1}(1 - \eta)\big).\]
    Next write
    \[\sum_{I \in [N]_{\text{lo}}^p} a_I(\sigma)c_I(\sigma) = \sqrt{\dfrac{N}{N - 1}}B_D(\sigma)^\top \nabla_yC(\sigma, \sigma).\]
    Let $v,v'\in \mathbb{R}^N$ be supported on $D(\sigma)$ and $L(\sigma)$ respectively with $v_i=\sigma_i$ on $D(\sigma)$ and $v_i'=\sigma_i$ on $L(\sigma)$. 
    A direct computation yields
    \begin{align*}
      \nabla_y C(\sigma,\sigma) &= p\big(1 - R_L(\sigma,\sigma)^{p-1}\big) v+p\big(1 - R_L(\sigma, \sigma)^{p - 1} - (p - 1)R_L(\sigma, \sigma)^{p-2}R_D(\sigma,\sigma)\big)v'.
    \end{align*}
    From Remark \ref{add:remark1}, $B_D(\sigma)^\top v=0$ and $B_D(\sigma)^\top v'=0$. These yield
    the second identity in \eqref{Equation:Ub_MGF_1}.

    For the last identity in \eqref{Equation:Ub_MGF_1}, write
    \begin{align}\label{add:newlemma:eq1}\sum_{I \in [N]_{\text{lo}}^p} c_I(\sigma)c_I(\sigma)^\top = B_D(\sigma)^\top \Big(\dfrac{\partial^2C(\sigma, \sigma)}{\partial x_s\partial y_t}\Big)_{s, t \in [N]} B_D(\sigma).
    \end{align}
    From a direct computation, it can be seen that the $D(\sigma)\times L(\sigma)$, $L(\sigma)\times D(\sigma)$, and $L(\sigma)\times L(\sigma)$-block submatrices of $(\partial_{x_s\partial y_t}C(\sigma,\sigma))_{s,t\in [N]}$ are of the form $b_{D,L}\sigma_D\sigma_L^\top,$ $b_{L,D}\sigma_L\sigma_D^\top$, and $d_{L,L}I_L+b_{L,L}\sigma_L\sigma_L^\top$, where $b_{D,L},b_{L,D},b_{L,L},$ and $d_{L,L}$ are scalars depending on $\sigma$ and $p$ and $\sigma_D$ and $\sigma_L$ are the restrictions of $\sigma$ to $D(\sigma)$ and $L(\sigma)$ respectively. As a consequence of Remark~\ref{add:remark1}, only the $D(\sigma)\times D(\sigma)$-block submatrix of $(\partial_{x_s\partial y_t}C(\sigma,\sigma))_{s,t\in [N]}$ can contribute to the right-hand side of \eqref{add:newlemma:eq1}. Since
    \begin{equation*}
        \Big(\dfrac{\partial^2C(\sigma, \sigma)}{\partial x_s\partial y_t}\Big)_{s, t \in D(\sigma)} = p\big(1 - \eta^{p - 1}\big)I_{D(\sigma)} + \dfrac{p(p - 1)}{N}\sigma_D\sigma_D^\top,
    \end{equation*}
    by using the second vanishing property in Remark \ref{add:remark1}, the proof of the third identity in \eqref{Equation:Ub_MGF_1} follows.
\end{proof}

\begin{proposition}
    If $\{J_I\}_{I \in [N]^p, N \in \mathbb{N}}$ satisfies \eqref{Equation:SSG}, then for all $\beta < 1/2p$,
    \begin{equation}\label{Equation:Ub_MGF}
        \limsup_{N \to \infty} \Lambda_N(\alpha, \beta) \leq \frac{\alpha^2}{2} - \frac{1}{2}\log(1 - 2\beta p).
    \end{equation}
\end{proposition}

\begin{proof}
    For $0 \leq \beta < 1/2p$, we linearize the quadratic term in the joint moment generating function by introducing $g \sim \texttt{Normal}(0, I_{N - 1})$ and get
    \begin{align*}
        \Lambda_N(\alpha, \beta) &= \dfrac{1}{N}\log\mathbb{E}_{\sigma, J}\big[\exp\big(\alpha H_N^J(\sigma) + \beta \|B(\sigma)^\top\nabla H_N^J(\sigma)\|_2^2\big)\big] \\
        &= \dfrac{1}{N}\log \mathbb{E}_\sigma\Big[\prod_{I \in [N]^p}M_{J_I}(\alpha a_I(\sigma) + \sqrt{2\beta}\langle b_I(\sigma), g\rangle)\Big] \\
        &\leq \dfrac{1}{N}\log \mathbb{E}_\sigma\Big[\exp\Big(\dfrac{1}{2}\sum_{I \in [N]^p} (\alpha a_I(\sigma) + \sqrt{2\beta}\langle b_I(\sigma), g\rangle)^2\Big)\Big] \\
        &= \dfrac{N\alpha^2}{2(N - 1)} - \dfrac{N - 1}{2N}\log(1 - 2\beta p),
    \end{align*}
    where the inequality follows from \eqref{Equation:SSG} and the sum can be evaluated via \eqref{Equation:Covariance_Ord_H}. The asymptote \eqref{Equation:Ub_MGF} then follows by taking $N \to \infty$.
    
    For $\beta = -\gamma < 0$, we see by \eqref{Equation:B_D_and_B} and the independence between $H_N^{\text{hi}}(\sigma)$ and $H_N^{\text{lo}}(\sigma)$,
    \begin{align}
        &\mathbb{E}_J\big[\exp\big(\alpha H_N^J(\sigma) - \gamma \|B(\sigma)^\top\nabla H_N^J(\sigma)\|_2^2\big)\big] \notag \\
        &\leq \mathbb{E}_J\big[\exp\big(\alpha H_N^J(\sigma) - \gamma \|B_D(\sigma)^\top\nabla H_N^J(\sigma)\|_2^2\big)\big] \notag \\
        &\leq \mathbb{E}_J\big[\exp\big(\alpha H_N^{\text{hi}}(\sigma)\big)\big] \sup_{v \in \mathbb{R}^{N - |L(\sigma)| - 1}}\mathbb{E}_J\big[\exp\big(\alpha H_N^\text{lo}(\sigma) - \gamma \|v + B_D(\sigma)^\top\nabla H_N^{\text{lo}}(\sigma)\|_2^2\big)\big]. \label{Equation:Ub_MGF_2}
    \end{align}
    For the first term in \eqref{Equation:Ub_MGF_2}, we combine \eqref{Equation:SSG} and the covariance profile \eqref{Equation:Ub_MGF_1} to yield
    \begin{equation}\label{Equation:Ub_MGF_3}
        \mathbb{E}_J\big[\exp\big(\alpha H_N^{\text{hi}}(\sigma)\big)\big] \leq \exp\Big(\dfrac{\alpha^2}{2}\sum_{I \in [N]_{\text{hi}}^p} a_I(\sigma)^2\Big) = \exp\Big(\frac{\alpha^2N^2(1 - s(\eta))}{2(N - 1)}\Big).
    \end{equation}
    For the second term in \eqref{Equation:Ub_MGF_2}, we apply the upper bound in Proposition \ref{Proposition:General_Bound} with the collection of deterministic coefficients $\{a_I(\sigma)\}_{I \in [N]_{\text{lo}}^p}$ and $\{c_I(\sigma)\}_{I \in [N]_{\text{lo}}^p}$. By the covariance profile \eqref{Equation:Ub_MGF_1}, we have the parameters 
    \[d_N = N - |L(\sigma)| - 1, \quad A_N = \dfrac{N^2}{N - 1}s(\eta), \quad \tau_N^2 = p\big(1 - \eta^{p - 1}\big) \leq p.\]
    Therefore, \eqref{Equation:Ub_MGF_2}, \eqref{Equation:Ub_MGF_3}, and \eqref{Equation:UB_General} together implies that
    \begin{align}
        &\mathbb{E}_J\big[\exp\big(\alpha H_N^J(\sigma) - \gamma \|B(\sigma)^\top\nabla H_N^J(\sigma)\|_2^2\big)\big] \notag \\
        &\leq \exp\Big(\dfrac{\alpha^2N^2(1 - s(\eta))}{2(N - 1)} + \frac{\alpha^2N^2s(\eta)}{2(N - 1)} - \frac{d_N}{2}\log\big(1 + 2\gamma p\big(1 - \eta^{p - 1}\big)\big) + \oldconstant{Constant:UB_MGF}\Delta_N\Big) \notag \\
        &= \exp\Big(\dfrac{N^2\alpha^2}{2(N - 1)} + \oldconstant{Constant:UB_MGF}\Delta_N\Big)\big(1 + 2\gamma p\big(1 - \eta^{p - 1}\big)\big)^{-d_N/2}. \label{Equation:Ub_MGF_4}
    \end{align}
    
    Next we claim that $\Delta_N = o_N(N)$. Note that for all $I \in [N]_{\text{lo}}^p$,
    \[|a_I(\sigma)|^2 \leq 2N^{1 - p}|\sigma_I|^2 \leq 2N^{1 - p}\big[\sqrt{N}^{p - 2}(\log N)^2\big]^2 = 2N^{-1}(\log N)^4 = o_N(1)\]
    and
    \begin{align*}
        \|c_I(\sigma)\|_2^2 &\leq \dfrac{\|\nabla\sigma\|_2^2}{N^{p - 1}} 
        = \dfrac{1}{N^{p - 1}}\Big\|\sum_{\ell = 1}^p\Big(\prod_{k \neq \ell} \sigma_{i_k}\Big)e_{i_\ell}\Big\|_2^2\\
        &\leq \dfrac{p}{N^{p - 1}}\sum_{\ell = 1}^p\Big(\prod_{k \neq \ell} \sigma_{i_k}\Big)^2\|e_{i_\ell}\|_2^2 \leq \dfrac{p^2(\log N)^2}{N} = o_N(1).
    \end{align*}
    Therefore, 
    \begin{align}
        \Delta_N &= \sum_{I \in [N]_{\text{lo}}^p} \big(a_I(\sigma)^4 + \|c_I(\sigma)\|_2^4\big) \notag \\
        &\leq \max_{I \in [N]_{\text{lo}}^p}\big(a_I(\sigma)^2 + \|c_I(\sigma)\|_2^2\big) \sum_{I \in [N]_{\text{lo}}^p} \big(a_I(\sigma)^2 + \|c_I(\sigma)\|_2^2\big) \notag \\
        &= o_N(1)\big(N^2(N - 1)^{-1}s(\eta) + p\big(1 - \eta^{p - 1}\big)d_N(\sigma)\big) = o_N(N), \label{Equation:Ub_MGF_5}
    \end{align}
    where the convergence is uniform over $\sigma \in \sqrt{N}\mathbb{S}^{N - 1}$. 
    
    Finally, we control the last term in \eqref{Equation:Ub_MGF_4}. For convenience, set $S(x) = 1 + 2\gamma p\big(1 - x^{p - 1}\big)$ for $x \in [0, 1]$. Observe that for any fixed subset $L \subseteq [N]$ with size $|L| \geq 1$, $\eta_L(\sigma): = N^{-1} \sum_{i \in L} \sigma_i^2$ obeys $\texttt{Beta}(|L|/2, (N - |L|/2))$. Hence,
    \begin{align}
        &\mathbb{E}_\sigma\big[\big(1 + 2\gamma p\big(1 - \eta(\sigma)^{p - 1}\big)^{-d_N/2}\big] \notag \\
        &= \sum_{\ell = 0}^{r_N}\sum_{L \subseteq [N], |L| = \ell}\mathbb{E}_\sigma\big[\mathbbm{1}_{\{L(\sigma) = L\}}S(\eta_L(\sigma))^{-(N - \ell - 1)/2}\big] \notag \\
        &\leq (1 + 2\gamma p)^{-(N - 1)/2} + \sum_{\ell = 1}^{r_N} \binom{N}{\ell} \mathbb{E}_{X \sim \texttt{Beta}(\ell/2, (N - \ell)/2)}\big[S(X)^{-(N - \ell - 1)/2}\big]. \label{Equation:Ub_MGF_6}
    \end{align}
    Here,
    \begin{align}
        &\mathbb{E}_{X \sim \texttt{Beta}(\ell/2, (N - \ell)/2)}\big[S(X)^{-(N - \ell - 1)/2}\big] \notag \\
        &= \dfrac{1}{B(\ell/2, (N - \ell)/2)}\int_0^1S(x)^{-(N - \ell - 1)/2}x^{\ell/2 - 1}(1 - x)^{(N - \ell)/2 - 1}\,dx \notag \\
        &= \dfrac{1}{B(\ell/2, (N - \ell)/2)}\int_0^1\Big(\dfrac{1 - x}{S(x)}\Big)^{(N - \ell - 1)/2}x^{\ell/2 - 1}(1 - x)^{-1/2}\,dx \notag \\
        &\leq (1 + 2\gamma p)^{-(N - \ell - 1)/2}\dfrac{B(\ell/2, 1/2)}{B(\ell/2, (N - \ell)/2)}, \label{Equation:Ub_MGF_7}
    \end{align}
    where the inequality holds since
    \[S(x) - (1 + 2\gamma p)(1 - x) = x\big(1 + 2\gamma p - 2\gamma p x^{p - 2}\big) \geq 0.\]
    Plugging \eqref{Equation:Ub_MGF_7} back in \eqref{Equation:Ub_MGF_6} yields
    \begin{equation}\label{Equation:Ub_MGF_8}
        \mathbb{E}_\sigma\big[S(\eta(\sigma))^{-d_N/2}\big] \leq  (1 + 2\gamma p)^{-(N - 1)/2}\Big[1 + \sum_{\ell = 1}^{r_N} \binom{N}{\ell}(1 + 2\gamma p)^{\ell/2}\dfrac{B(\ell/2, 1/2)}{B(\ell/2, (N - \ell)/2)}\Big].
    \end{equation}
    Now, by Stirling's formula,  uniformly for $1 \leq \ell \leq r_N$,
    \[\log\binom{N}{\ell} = O_N(\ell\log N) \quad \text{and} \quad \log\Big(\frac{B(\ell/2,1/2)}{B(\ell/2,(N - \ell)/2)}\Big) = O_N(\ell\log N + \log N).\]
    Since $r_N\log N = O_N(N/\log N) = o_N(N)$, the  square bracket in \eqref{Equation:Ub_MGF_8} is $e^{o_N(N)}$. Combining \eqref{Equation:Ub_MGF_4}, \eqref{Equation:Ub_MGF_5}, \eqref{Equation:Ub_MGF_6}, and \eqref{Equation:Ub_MGF_8} leads to the desired upper bound.
\end{proof}

\subsection{Lower Bound}

\begin{proposition}
    If $p \geq 3$ and $\{J_I\}_{I \in [N]^p, N \in \mathbb{N}}$ satisfies \eqref{Equation:SSG}, then 
    \begin{equation}
        \liminf_{N \to \infty} \Lambda_N(\alpha, \beta) \geq \begin{cases}\label{Equation:LB_MGF}
            \frac{\alpha^2}{2} - \frac{1}{2}\log(1 - 2\beta p), & \beta < \frac{1}{2p}, \\
            \infty, & \beta \geq \frac{1}{2p}.
        \end{cases}
    \end{equation}
\end{proposition}

\begin{proof}
    Set the remainder $R_I(z) = \log M_{J_I}(z) - z^2/2$. Since $\mathbb{E}[J_I] = 0$, $\mathbb{E}[J_I^2] = 1$, there exists $\delta, C > 0$ such that
    \begin{equation}\label{Equation:MGF_Control}
        \sup_{N \in \mathbb{N}}\sup_{I \in [N]^p}|R_I(z)| \leq C|z|^3 \quad \text{for all} \quad |z| \leq \delta.
    \end{equation}
    First, we argue for $\beta \geq 0$. By introducing an independent random vector $g \sim \texttt{Normal}(0, I_{N - 1})$ and denoting $z_I(\sigma) = \alpha a_I(\sigma) + \sqrt{2\beta}\langle b_I(\sigma), g\rangle \in \mathbb{R}$, we again linearize the quadratic term in the joint moment generating function: 
    \begin{align}
        \Lambda_N(\alpha, \beta) &= \dfrac{1}{N}\log \mathbb{E}_{\sigma, g}\Big[\exp\Big(\sum_{I \in [N]^p} \log M_{J_I}(z_I(\sigma))\Big)\Big] \notag \\
        &= \dfrac{1}{N}\log \mathbb{E}_{\sigma, g}\Big[\exp\Big(\dfrac{1}{2}\sum_{I \in [N]^p} z_I(\sigma)^2 + \sum_{I \in [N]^p} R_I(z_I(\sigma))\Big)\Big] \notag \\
        &= \dfrac{N\alpha^2}{2(N - 1)} + \dfrac{1}{N}\log \mathbb{E}_{\sigma, g}\Big[\exp\Big(\beta p\|g\|_2^2 + \sum_{I \in [N]^p}R_I(z_I(\sigma))\Big)\Big], \label{Equation:MGF_Lower_1}
    \end{align}
    where \eqref{Equation:MGF_Lower_1} holds by the covariance profile \eqref{Equation:Covariance_Ord_H}. Consider for $M \geq 1$ the events
    \[\mathcal{E}_N^1 = \big\{\sigma \in \sqrt{N}\mathbb{S}^{N - 1} \colon \|\sigma\|_\infty \leq \log N\} \quad \text{and} \quad \mathcal{E}_N^2(M) = \{g \in \mathbb{R}^{N - 1} \colon \|g\|_\infty \leq M\}.\]
    Note that $\mathbb{P}_{\sigma}(\mathcal{E}_N^1) \to 1$ as $N \to \infty$. We claim that the argument of the remainder $|z_I(\sigma)| = o_N(1)$ on the event $\mathcal{E}_N^1 \cap \mathcal{E}_N^2(M)$. Indeed, on $\mathcal{E}_N^1 \cap \mathcal{E}_N^2(M)$,
    \[|\alpha a_I(\sigma)| = \dfrac{|\alpha||\sigma_I|}{N^{(p - 1)/2}} \leq \dfrac{|\alpha|\|\sigma\|_\infty^p}{N^{(p - 1)/2}} \leq \dfrac{|\alpha|(\log N)^p}{N^{(p - 1)/2}} = o_N(1)\]
    and also
    \begin{align*}
        \big|\sqrt{2\beta}\langle b_I(\sigma), g\rangle \big| &\leq \dfrac{\sqrt{2\beta}}{N^{(p - 1)/2}}\Big|\sum_{j = 1}^p \sigma_{I \setminus \{i_j\}}(B(\sigma)g)_{i_j}\Big| \\
        &\leq \dfrac{\sqrt{2\beta}p\|\sigma\|_\infty^{p - 1}\|B(\sigma)g\|_\infty}{N^{(p - 1)/2}} \\
        &\leq \dfrac{\sqrt{2\beta}p(\log N)^{p - 1}\|g\|_2}{N^{(p - 1)/2}} = O\Big(\dfrac{(\log N)^{p - 1}}{N^{(p - 2)/2}}\Big) = o_N(1),
    \end{align*}
    where we recall here that $p \geq 3$. Since $|z_I(\sigma)| = o_N(1)$, we have by \eqref{Equation:MGF_Control} that for $N$ large enough, there exists constant $\newconstant\label{Constant:Lower_1} = \oldconstant{Constant:Lower_1}(\alpha, \beta) > 0$ such that
    \begin{align*}
        &\Big|\sum_{I \in [N]^p} R_I(\alpha a_I(\sigma) + \sqrt{2\beta}\langle b_I(\sigma), g\rangle)\Big| \\
        &\leq \oldconstant{Constant:Lower_1}\sum_{I \in [N]^p} \big(|a_I(\sigma)|^3 + |\langle b_I(\sigma), g\rangle|^3\big) \\
        &\leq \oldconstant{Constant:Lower_1}\max_{I \in [N]^p}  \big(|a_I(\sigma)| + |\langle b_I(\sigma), g\rangle|\big) \sum_{I \in [N]^p}  \big(|a_I(\sigma)|^2 + |\langle b_I(\sigma), g\rangle|^2\big) \\
        &= o_N(1)\Big(\dfrac{N^2}{N - 1} + p\|g\|_2^2\Big) = o_N(N),
    \end{align*}
    where the last equality holds by the covariance profile \eqref{Equation:Covariance_Ord_H} and all estimates are uniform on the event $\mathcal{E}_N^1 \cap \mathcal{E}_N^2(M)$. We then see from \eqref{Equation:MGF_Lower_1} that
    \begin{align*}
        \liminf_{N \to \infty}\Lambda_N(\alpha, \beta) &\geq \dfrac{\alpha^2}{2} + \liminf_{N \to \infty} \dfrac{1}{N}\log \mathbb{E}_{\sigma, g}\big[e^{\beta p\|g\|_2^2 + o_N(N)}\mathbbm{1}_{\mathcal{E}_N^1 \cap \mathcal{E}_N^2(M)}\big] \\
        &= \dfrac{\alpha^2}{2} + \liminf_{N \to \infty} \dfrac{1}{N}\log\Big(\dfrac{1}{\sqrt{2\pi}}\int_{-M}^M e^{(\beta p - \frac{1}{2})x^2}\,dx\Big)^N \\
        &= \dfrac{\alpha^2}{2} + \log\Big(\dfrac{1}{\sqrt{2\pi}}\int_{-M}^M e^{(\beta p - \frac{1}{2})x^2}\,dx\Big)
    \end{align*}
    and may deduce our desired lower bound \eqref{Equation:LB_MGF} by taking $M \to \infty$.

    For $\beta = -\gamma < 0$,  we apply the lower bound in Proposition \ref{Proposition:General_Bound} with the collection of deterministic coefficients $\{a_I(\sigma)\}_{I \in [N]^p}$ and $\{b_I(\sigma)\}_{I \in [N]^p}$. By the covariance profile \eqref{Equation:Covariance_Ord_H}, we have the parameters 
    \[d_N = N - 1, \quad A_N = \dfrac{N^2\alpha^2}{N - 1}, \quad \tau_N^2 = p.\]
    Moreover, on the event
    \[\mathcal{E}_N = \{\sigma \in \sqrt{N}\mathbb{S}^{N - 1} \colon \|\sigma\|_\infty \leq \log N\}\]
    which satisfies $\mathbb{P}_\sigma(\mathcal{E}_N) \to 1$, we have, by a similar argument as in the case $\beta \geq 0$, the bound
    \[\rho_N = O_N\Big(\dfrac{(\log N)^p}{N^{(p - 1)/2}}\Big) = o_N(1) \quad \text{and} \quad \Delta_N = O_N(\rho_N^2N) = O_N\Big(\dfrac{(\log N)^{2p}}{N^{p - 2}}\Big) = o_N(1).\]
    Therefore, \eqref{Equation:LB_General} implies that
    \begin{align*}
        \Lambda_N(\alpha, \beta) &\geq \dfrac{1}{N}\log \mathbb{E}_\sigma\Big[\mathbbm{1}_{\mathcal{E}_N}\mathbb{E}_g\Big[\exp\Big(\alpha\sum_{I \in [N]^p} a_I(\sigma)J_I - \gamma\Big\|\sum_{I \in [N]^p} b_I(\sigma)J_I\Big\|_2^2\Big)\Big] \\
        &= \dfrac{N\alpha^2}{2(N - 1)} - \dfrac{N - 1}{2N}\log(1 + 2\gamma p) + \dfrac{1}{N}\log\mathbb{P}_\sigma(\mathcal{E}_N) + o_N(1)
    \end{align*}
    and this concludes the proof.
\end{proof}

\section{Proof of Main Results}\label{Section:Main}

\begin{proof}[\bf Proof of Theorem \ref{Theorem:LDP}] Recall from \eqref{Equation:Def_X} and \eqref{Equation:Def_y} the following notations:
\[X_N^J(\sigma) = \frac{H_N^J(\sigma)}{\sqrt{N(N-1)}} \quad \text{and} \quad y_N^J(\sigma) = \frac{\|B(\sigma)^\top\nabla H_N^J(\sigma)\|_2^2}{N}.\]
By \eqref{Equation:Convergence_of_MGF}, the limiting logarithmic moment generating function of $(X_N^J(\sigma), y_N^J(\sigma))$ is
\[\Lambda(\alpha, \beta) = \lim_{N \to \infty} \Lambda_N(\alpha, \beta) = \begin{cases}
    \frac{\alpha^2}{2} - \frac{1}{2}\log(1 - 2\beta p), & \beta < \frac{1}{2p}, \\
    \infty, & \beta \geq \frac{1}{2p}.
\end{cases}\]
Its effective domain is $\mathcal{D}_\Lambda = \mathbb{R} \times (-\infty,1/(2p))$. The function $\Lambda$ is differentiable on $\mathcal{D}_\Lambda$, lower semicontinuous on $\mathbb{R}^2$, and
\[\|\nabla\Lambda(\alpha,\beta)\|_2 = \Big\|\Big(\alpha, \dfrac{p}{1 - 2p\beta}\Big)\Big
\|_2 \to \infty \quad \text{as} \quad \beta \uparrow \dfrac{1}{2p},\]
so $\Lambda$ is essentially smooth (see \cite[Definition 2.3.5]{dembo2010large}). Therefore, by G\"{a}rtner-Ellis theorem (see \cite[Theorem 2.3.6]{dembo2010large}), $(X_N^J(\sigma), y_N^J(\sigma))$ satisfies the LDP at speed $N$ with the good rate function
\[\Lambda^\ast(a, b) = \sup_{\alpha \in \mathbb{R}, \beta < 1/(2p)}\Big[(\alpha a + \beta b) - \Big(\dfrac{\alpha^2}{2} - \dfrac{1}{2}\log(1 - 2\beta p)\Big)\Big]. \]
For $b > 0$, the optimizers are $\alpha = a$ and $\beta = (b-p)/(2pb)$, while for $b \leq 0$ the supremum is infinite. Hence $\Lambda^\ast = \mathfrak{I}_2$, which proves Theorem \ref{Theorem:LDP}.
\end{proof}

\begin{proof}[\bf Proof of Theorem \ref{Theorem:Universality_Main}]  
    Since $Z_N := Z/\sqrt{N-1}$ is independent of $(X_N^J(\sigma), y_N^J(\sigma))$ and satisfies the LDP at speed $N$ with the good rate function $c^2/2$, the sequence
    \begin{equation}\label{Equation:Joint_3}
        \big(X_N^J(\sigma), y_N^J(\sigma), Z_N\big) \in \mathcal{D} := \mathbb{R} \times  [0,\infty) \times \mathbb{R}
    \end{equation}
    satisfies the LDP at speed $N$ with the good rate function
    \[\mathfrak{I}_3(a, b, c) = \mathfrak{I}_2(a, b) + \dfrac{c^2}{2} = \begin{cases}
        \frac{a^2}{2} + \frac{c^2}{2} + \frac{b - p}{2p} + \frac{1}{2}\log\big(\frac{p}{b}\big), & b > 0, \\
        \infty, & b \leq 0.
    \end{cases}\]
    
    Next, we compute the exponential rate of $\ell_N(x,\delta)$, as defined in \eqref{Equation:ell_N} for a fixed $x \in \mathbb{R}$ and $\delta > 0$. Set
    \[T(a,c) = \sqrt{\dfrac{p}{p - 1}}\big(\sqrt{t}\,a + \sqrt{1 - t}\,c\big).\]
    The set $\{(a,b,c) \in \mathcal{D} \colon |T(a, c) - x|\leq \delta\}$  is closed in $\mathcal{D}$, and the function $(a,b,c) \mapsto -(tb)/(2p(1 - t))$ is continuous and bounded above on $\mathcal{D}$. Therefore, by Varadhan's lemma, more precisely, see \cite[Exercise 4.3.11]{dembo2010large}, we obtain 
    \begin{align}
        \limsup_{N\to\infty}\frac{1}{N}\log\ell_N(x, \delta) &= \limsup_{N\to\infty}\frac{1}{N}\log \mathbb{E}_{\sigma, J, Z}\Big[\exp\Big(N\Big(-\frac{ty_N^J(\sigma)}{2p(1 - t)}\Big)\Big) \mathbbm{1}_{\big\{\big|T(X_N^J(\sigma), Z_N) - x \big| \leq \delta \big\}}\Big] \notag \\
        &\leq \sup_{\substack{|T(a,c)-x|\leq \delta \\ b > 0}} \Big\{-\frac{tb}{2p(1-t)} - \mathfrak{I}_3(a,b,c) \Big\} \notag \\
        &= \frac{1}{2}\log(1 - t) - \frac{p - 1}{2p}(|x|-\delta)_+^2. \label{Equation:Varadhan_Limit}
    \end{align}
Indeed, the optimization over $b$ is attained at $b = p(1-t)$ and has value $\log(1 - t)/2$, while
\[\inf_{|T(a, c) - x| \leq \delta} \frac{a^2 + c^2}{2} = \frac{p - 1}{2p}(|x| - \delta)_+^2.\]
Similarly, applying Varadhan's lemma to the open set $\{(a, b, c)\in \mathcal{D} \colon |T(a, c) - x| < \delta\}$, we deduce  
\[\liminf_{N\to\infty}\frac{1}{N}\log\ell_N(x,\delta) \geq \sup_{\substack{|T(a,c) - x| < \delta \\ b > 0}} \Big\{ -\frac{tb}{2p(1-t)} -\mathfrak I_3(a,b,c) \Big\},\]
which has the same optimal value as in \eqref{Equation:Varadhan_Limit}. Therefore, 
\begin{equation}\label{Equation:Ell_Exact_Limit}
    \lim_{N \to \infty} \frac{1}{N}\log\ell_N(x,\delta) = \frac{1}{2}\log(1 - t) - \frac{p-1}{2p}(|x|-\delta)_+^2.
\end{equation}

Recall the quantity $\mathcal{L}_N(u)$ from \eqref{Equation:L_N(u)}. Combining Proposition \ref{Theorem:Upper} with \eqref{Equation:Ell_Exact_Limit} yields
\[\limsup_{N\to\infty}\frac{1}{N} \log\mathcal L_N(u) \leq \frac{1}{2}\log(1 - t) + \sup_{x \leq 2u/ E_\infty} \Big(\Omega(x) - \frac{p-1}{2p}x^2 \Big).\]
Indeed, the functions 
\[F_\delta(x) = \Omega(x) - \dfrac{p - 1}{2p}(|x|-\delta)_+^2\] 
converge uniformly to $F_0(x)$ on a compact set containing all their maximizers for $0 \leq \delta \leq 1$. Similarly, Proposition~\ref{Theorem:Lower} and
\eqref{Equation:Ell_Exact_Limit} give
\[\liminf_{N\to\infty}\frac{1}{N}\log\mathcal L_N(u) \geq \frac{1}{2}\log(1-t) + \sup_{x < 2u/ E_\infty} \Big( \Omega(x)-\frac{p-1}{2p}x^2 \Big).\]
Since $F_0(x)$ is continuous, its suprema on $x < 2u/ E_\infty$ and $x\leq 2u/ E_\infty$ are the same. It follows
\begin{equation}\label{Equation:Final_LN_Limit}
    \lim_{N\to\infty}\frac{1}{N}\log\mathcal L_N(u) = \frac{1}{2}\log(1-t) + \sup_{x\le 2u/ E_\infty} \Big( \Omega(x)-\frac{p-1}{2p}x^2 \Big).
\end{equation}

Combining \eqref{Equation:Final_LN_Limit} with the Kac-Rice formula \eqref{Equation:Kac_Rice}, we obtain
\begin{equation}\label{Equation:Final_Variational_Formula}
    \lim_{N \to \infty}\frac{1}{N}\log\mathbb{E}_{J,G}\big[\text{Crt}_t((-\infty, u))\big] = \frac{1}{2}\log(p - 1) + \frac{1}{2} + \sup_{x \leq 2u/E_\infty} \Big(\Omega(x) - \frac{p-1}{2p}x^2 \Big).
\end{equation}
The right-hand side of \eqref{Equation:Final_Variational_Formula} does not depend on $t$, and therefore equals the corresponding expression at $t = 0$. This proves Theorem \ref{Theorem:Universality_Main}.
\end{proof}

\begin{appendices}

\section{Useful Inequalities}

\begin{lemma}[Gaussian Brascamp-Lieb inequality]\label{lem:gaussian-BL}
    Let $u_1,\ldots,u_m \in \mathbb{R}^d$ be unit vectors and let $\lambda_1, \dots, \lambda_m > 0$ satisfy
    \[\sum_{j = 1}^m\lambda_j u_ju_j^\top = I_d.\]
    If $g \sim \texttt{Normal}(0,I_d)$ and $Z \sim \texttt{Normal}(0, 1)$, then
    for any nonnegative measurable functions $f_j\colon \mathbb{R} \to [0, \infty)$ with $\mathbb{E}[f_j(Z)] < \infty$,
    \begin{equation}\label{Equation:GBL}
        \mathbb{E}\Big[\prod_{j = 1}^m f_j(\langle u_j, g\rangle)^{\lambda_j}\Big] \leq \prod_{j = 1}^m \mathbb{E}[f_j(Z)]^{\lambda_j}.
    \end{equation}
\end{lemma}

\begin{proof}
     This is a rank-one specialization of \cite[Theorem 4]{chen2015improved}.
\end{proof}

\begin{lemma}
    If $X_1, \dots, X_n$ are (not necessarily independent) random variables satisfying \eqref{Equation:SG} with uniform $\kappa_{\text{SG}}$, there exists an absolute constant $\newconstant\label{Constant:Sup_SG} > 0$ such that for all $k \in \mathbb{N}$,
    \begin{equation}\label{Equation:Sup_SG}
        \mathbb{E}\Big[\Big|\max_{1 \leq j \leq n} X_j\Big|^k\Big]^{1/k} \leq \oldconstant{Constant:Sup_SG}\kappa_{\text{SG}}\sqrt{\log n + k}.
    \end{equation}
\end{lemma}

\begin{proof}
    By the condition \eqref{Equation:SG} and union bound, we have
    \[\mathbb{P}\Big(\Big|\max_{1 \leq j \leq n} X_j\Big| \geq h\Big) \leq 2ne^{-h^2/2\kappa_{\text{SG}}^2}.\]
    Define $a := (2\kappa_{\text{SG}}^2\log(2n))^{1/2}$ and $Y = (|\max_{1 \leq j \leq n}X_j| - a)_+$. Then,
    \begin{align*}
        \mathbb{E}[|Y|^k] &= k\int_0^\infty h^{k - 1}\mathbb{P}(|Y| \geq h)\,dh \\
        &= k\int_0^\infty h^{k - 1}\mathbb{P}\Big(\Big|\max_{1 \leq j \leq n} X_j\Big| \geq a + h\Big)\,dh \notag \\
        &\leq 2nk\int_0^\infty h^{k - 1}e^{-(a + h)^2/2\kappa_{\text{SG}}^2}\,dh \leq k\int_0^\infty h^{k - 1}e^{-h^2/2\kappa_{\text{SG}}^2}\,dh = k\cdot 2^{k/2 - 1}\kappa_{\text{SG}}^k\Gamma\Big(\frac{k}{2}\Big).
    \end{align*}
    Since $\Gamma(k/2)^{1/k} = O_k\big(\sqrt{k}\big)$ as $k \to \infty$, we have
    \[\Big\|\max_{1 \leq j \leq n} X_j\Big\|_{\mathcal{L}^k} \leq a + \|Y\|_{\mathcal{L}^k} = O\big(\kappa_{\text{SG}}\sqrt{\log n + k}\big)\]
    and this ends the proof for the lemma.
\end{proof}

\begin{lemma}[{\cite[Lemma 6.3]{arous2001aging}}]\label{Lemma:Tail_of_Norm}
    For $h \geq 8$ and all $N \in \mathbb{N}$,
    \begin{equation}\label{Equation:Tail_of_GOE}
        \mathbb{P}(\|\text{GOE}_N\|_{\text{op}} \geq h) \leq e^{-Nh^2/9}.
    \end{equation}
\end{lemma}

\end{appendices}

\bibliographystyle{acm}

{\footnotesize\bibliography{reference}}
\end{document}